\documentclass[preprint,12pt]{elsarticle}

\usepackage{amssymb}
\usepackage{amsthm}
\usepackage{amsmath}

\makeatletter
	
	\@addtoreset{equation}{section}
\makeatother

\journal{}
\makeatletter
\def\ps@pprintTitle{
\let\@oddhead\@empty\let\@evenhead\@empty\def\@oddfoot{}\def\@evenfoot{}}
\makeatother

\newtheorem{thm}{Theorem}[section]
\newtheorem{prop}[thm]{Proposition}
\newtheorem{lem}[thm]{Lemma}
\newtheorem{cor}[thm]{Corollary}
\theoremstyle{definition}
\newtheorem{dfn}[thm]{Definition}
\newtheorem{rem}[thm]{Remark}

\newcommand{\R}{\mathbb{R}}
\newcommand{\C}{\mathbb{C}}
\newcommand{\N}{\mathbb{N}}

\newcommand{\mean}[3]{\mathbb{E}^{#1}_{#2}\left[ #3 \right]}
\newcommand{\sle}{\mathrm{SLE}}
\newcommand{\skle}{\mathrm{SKLE}}

\newcommand{\bmd}{\mathrm{BMD}}
\newcommand{\disk}{\mathbb{D}}
\newcommand{\uhp}{\mathbb{H}}
\newcommand{\Slit}{\mathsf{Slit}}
\newcommand{\slit}{\mathbf{s}}
\DeclareMathOperator{\dist}{dist}
\DeclareMathOperator{\res}{Res}
\DeclareMathOperator{\hcap}{hcap}
\DeclareMathOperator{\diam}{diam}
\DeclareMathOperator{\ins}{ins}

\begin{document}

\begin{frontmatter}



\title{Chordal Komatu--Loewner equation \\
for a family of continuously growing hulls}


\author[kyoto]{Takuya Murayama}
\ead{murayama@math.kyoto-u.ac.jp}
\address[kyoto]{Department of Mathematics, Kyoto University, Kyoto 606-8502, Japan}

\begin{abstract}
In this paper, we discuss the chordal Komatu--Loewner equation
on standard slit domains in a manner applicable
not just to a simple curve but also a family of continuously growing hulls.
Especially a conformally invariant characterization of
the Komatu--Loewner evolution is obtained.
As an application, we prove a sort of conformal invariance, or locality,
of the stochastic Komatu--Loewner evolution
$\mathrm{SKLE}_{\sqrt{6}, -b_{\mathrm{BMD}}}$ in a fully general setting,
which solves an open problem posed by Chen, Fukushima and
Suzuki~[Stochastic Komatu--Loewner evolutions and SLEs,
Stoch.\ Proc.\ Appl.\ \textbf{127} (2017), 2068--2087].
\end{abstract}


\end{frontmatter}


\section{Introduction}
\label{sec:intro}

The Komatu--Loewner equation is an extension of the celebrated
Loewner equation to multiply connected domains.
This equation describes the time-evolution of increasing subsets
of multiply connected domains, called growing hulls,
and was rigorously obtained in the previous studies~\cite{BF08, CFR16, CF18}
when the family of growing hulls consist of a trace of a simple curve.
In this paper, we shall give a systematic treatment of this equation
for a family of growing hulls which are not necessarily induced by a simple curve.
In order to describe mathematical details,
we begin to recall the Loewner theory briefly.
The reader can consult \cite{La05} for further detail.

We denote by $\uhp$ the upper half-plane $\{z \in \C; \Im z>0\}$.
Let $\gamma \colon [0, t_{\gamma}) \to \overline{\uhp}$ be a simple curve
with $\gamma(0) \in \partial \uhp$ and $\gamma(0, t_{\gamma}) \subset \uhp$.
For each $t \geq 0$, there exists a unique conformal map $g_t$
from $\uhp \setminus \gamma(0,t]$ onto $\uhp$
with the \emph{hydrodynamic normalization}
\[
g_t(z)=z+\frac{a_t}{z}+o(z^{-1}), \quad z \to \infty,
\]
for some constant $a_t>0$.
This is a version of Riemann's mapping theorem.
If we reparametrize $\gamma$ so that $a_t=2t$
(as mentioned later in Section~\ref{subsec:deduction}),
then we obtain the \emph{chordal Loewner equation}
\begin{equation} \label{eq:Loewner}
\frac{d}{dt}g_t(z)=\frac{2}{g_t(z)-\xi(t)}, \quad g_0(z)=z \in \uhp,
\end{equation}
where $\xi(t)=g_t(\gamma(t)):=\lim_{z \to \gamma(t)}g_t(z) \in \partial \uhp$.
We call $\xi$ the \emph{driving function} of $\{g_t\}$.

Since \eqref{eq:Loewner} is an ODE satisfying the local Lipschitz condition,
the solution $g_t(z)$ to \eqref{eq:Loewner} uniquely exists
up to its explosion time $t_z$.
If we set $F_t:=\{z \in \uhp; t_z \leq t\}$,
then $F_t$ must be the complement of the domain of definition of $g_t$,
that is, $F_t=\gamma(0, t]$.
Thus the information on the curve $\gamma$ is fully encoded into
the driving function $\xi(t)$ via the Loewner equation.
More generally, we can consider \eqref{eq:Loewner} driven by
any continuous function $\xi$.
Even in this case, the solution $g_t(z)$ defines a unique conformal map
$g_t \colon \uhp \setminus F_t \to \uhp$ with the hydrodynamic normalization,
though the resulting family $\{F_t\}$ is not necessarily a simple curve
but a family of bounded sets called growing hulls.
$\{F_t\}$, $\{g_t\}$ or the couple $(g_t, F_t)$ is called
the \emph{Loewner evolution driven by $\xi$}.
In the theory of conformal maps,
$\{g_t\}$ is usually called the \emph{Loewner chain}.

Schramm~\cite{Sc00} used the Loewner equation \eqref{eq:Loewner}
to define the \emph{stochastic Loewner evolution} (SLE).
For $\kappa>0$, $\sle_{\kappa}$ is the random Loewner evolution
driven by $\xi(t)=\sqrt{\kappa}B_t$,
where $B_t$ is the one-dimensional standard Brownian motion~(BM).
Schramm's original aim was to describe
the scaling limit of two-dimensional lattice models in statistical physics.
$\sle_{\kappa}$ was actually proven to be the scaling limit of some models
according to the value of $\kappa$.
For individual models, we refer the reader to \cite[Section~2.5]{Ka15}
and the references therein.
In addition, recent studies such as \cite{Fr10} reveal
the relation between the Loewner equation and integrable systems.
We therefore have much interest in the Loewner theory
from various points of view.

As seen, for example, from the usage of Riemann's mapping theorem above,
the simple connectivity of $\uhp$ is crucial to the Loewner theory.
Thus it is not straightforward to extend the Loewner equation
to multiply connected domains (or to Riemann surfaces).
This problem was originally proposed by Komatu~\cite{Ko50},
who obtained primary expression of corresponding equations
on special multiply connected domains.
After more than fifty years, Bauer and Friedrich~\cite{BF08} established
its definitive expression by means of the Green function and harmonic measures,
a standard way in complex analysis used by \cite{Ko50}.
Lawler~\cite{La06} then gave a probabilistic comprehension of the equation
in terms of the \emph{excursion reflected Brownian motion} (ERBM).
The idea provided in \cite{La06} was implemented by Drenning~\cite{Dr11}
later in some detail.
Motivated by \cite{BF08} and \cite{La06},
Chen, Fukushima and Rohde~\cite{CFR16} adopted the notion of
the \emph{Brownian motion with darning} (BMD)
to fill missing arguments in the existing proofs.

We now describe the framework where our domain has multiple connectivity.
Fix a positive integer $N$.
Let $C_j \subset \uhp$, $1 \leq j \leq N$, be mutually disjoint horizontal slits,
that is, segments parallel to the real axis.
We call $D:=\uhp \setminus \bigcup_{j=1}^N C_j$ a \emph{standard slit domain}.
Any $N$-connected domain is conformally equivalent to some standard slit domain.
The case of parallel slit plane, namely,
the whole plane $\C$ deleted by some parallel slits,
is typically treated in some textbooks, and in the present case 
the proof is almost the same as explained in \cite[Section~2.2]{BF08}.

Let $\gamma \colon [0, t_{\gamma}) \to \overline{D}$ be a simple curve
with $\gamma(0) \in \partial \uhp$ and $\gamma(0, t_{\gamma}) \subset D$.
For each $t \geq 0$, there exists a unique conformal map $g_t$
from $D \setminus \gamma(0,t]$ onto another standard slit domain $D_t$
with the hydrodynamic normalization.
After the same reparametrization of $a_t$,
$g_t(z)$ satisfies the \emph{chordal Komatu--Loewner equation}
(\cite[Theorem~3.1]{BF08}, \cite[Theorem~9.9]{CFR16})
\begin{equation} \label{eq:KL}
\frac{d}{dt}g_{t}(z) = -2\pi \Psi_{D_t}(g_{t}(z),\xi(t)), \quad g_{0}(z)=z \in D,
\end{equation}
where $\xi(t)=g_t(\gamma(t)) \in \partial \uhp$.
$\Psi_{D_t}(\cdot, \xi_0)$, $\xi_0 \in \R$, is the conformal map on $D_t$
defined in Section~\ref{subsec:BMD_conf}.

\begin{figure}
\centering
\includegraphics[width=12cm]{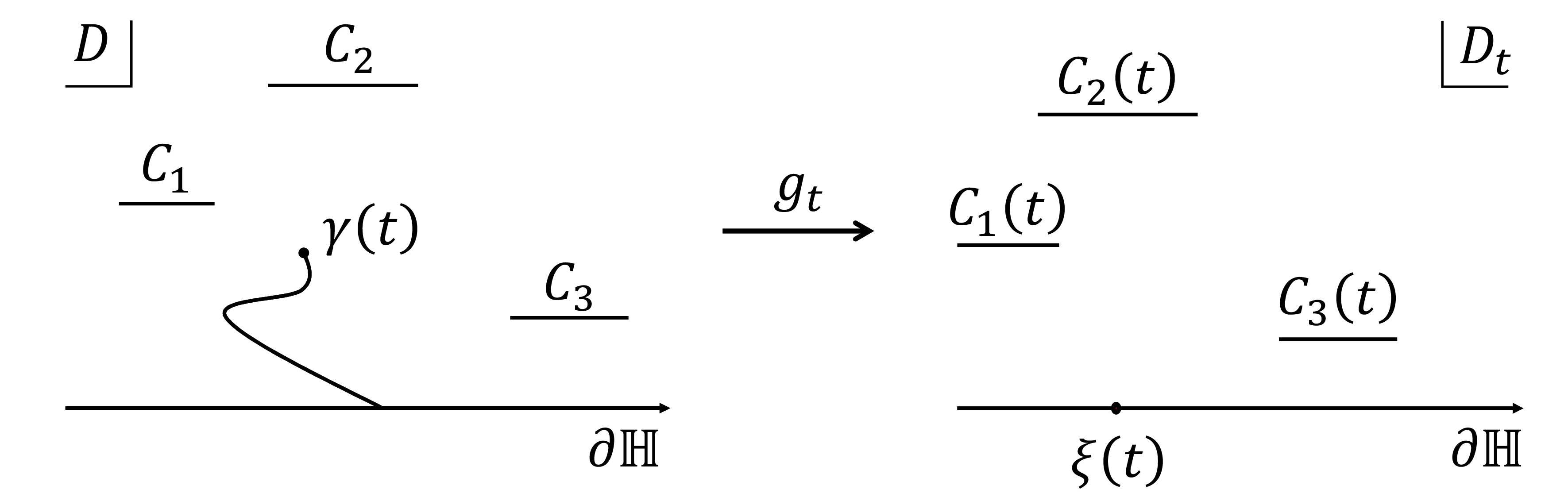}
\caption{Conformal map $g_t$}
\end{figure}

Here \eqref{eq:KL} differs from \eqref{eq:Loewner} in that
the image $D_t$ differs from $\uhp$ and varies as time passes.
Let $C_j(t)$ be the $j$-th slit of $D_t$ so that $C_j(0)=C_j$.
The left and right endpoints of $C_j(t)$ are denoted by $z_j(t)=x_j(t)+iy_j(t)$
and $z^r_j(t)=x^r_j(t)+iy_j(t)$, respectively.
These endpoints then satisfy the \emph{Komatu--Loewner equation for slits}
(\cite[Theorem~4.1]{BF08}, \cite[Theorem~2.3]{CF18})
\begin{equation} \label{eq:KLs}
\begin{split}
\frac{d}{dt}y_{j}(t) = -2\pi \Im \Psi_{D_t}(z_{j}(t),\xi(t)), \\
\frac{d}{dt}x_{j}(t) = -2\pi \Re \Psi_{D_t}(z_{j}(t),\xi(t)), \\
\frac{d}{dt}x^{r}_{j}(t) = -2\pi \Re \Psi_{D_t}(z^{r}_{j}(t),\xi(t)).
\end{split}
\end{equation}
Hence the motion of $D_t$ is described by \eqref{eq:KLs}
in terms of those of the slits $C_j(t)$.

Once we get \eqref{eq:KL} and \eqref{eq:KLs},
the initial value problem for them, as done for \eqref{eq:Loewner},
is a natural question.
Namely, for a given continuous function $\xi$,
we look for the solution to \eqref{eq:KL} and \eqref{eq:KLs}
and then obtain a family $\{F_t\}$ of growing hulls.
We shall explain the actual procedure in Section~\ref{subsec:initKL}.
As a result, \eqref{eq:KL} generates a family $\{g_t\}$
of conformal maps and $\{F_t\}$ of growing hulls.
They are called the \emph{Komatu--Loewner evolution
driven by $\xi$}.
Let us call $\{g_t\}$ the Komatu--Loewner chain as well in this paper.
In addition, Chen and Fukushima~\cite{CF18} defined
the \emph{stochastic Komatu--Loewner evolution}~(SKLE)
with the random driving function $\xi$ given by the system of SDEs
\eqref{eq:driv} and \eqref{eq:slit},
based on the discussion in \cite[Section~5]{BF08}.
Its relation to SLE was also examined by Chen, Fukushima and Suzuki~\cite{CFS17}.

In such a research on SKLE, the trouble often arises
concerning the ``transformation of the Komatu--Loewner chains.''
Here by the term ``transformation'' we mean the following situation:
Let $(g_t, F_t)$ be the Komatu--Loewner evolution in a standard slit domain $D$
and $\tilde{D}$ be another slit domain with $F_t \subset \tilde{D}$.
The degree of connectivity of $\tilde{D}$ can be different from that of $D$.
There is then a unique conformal map $\tilde{g}_t$ from $\tilde{D} \setminus F_t$
onto a slit domain with the hydrodynamic normalization
by Proposition~\ref{prop:canonical}.
We expect $(\tilde{g}_t, F_t)$ to be the Komatu--Loewner evolution in $\tilde{D}$,
that is, generated by the equation (modulo time-change).
This fact however needs proof since we have deduced the equation
only for a simple curve, not for a family of growing hulls.
From this standpoint, we can say that \cite{CFS17} established
exactly the transformation of chains with $\tilde{D}=\uhp$
by the hitting time analysis for BM.
This method is successful but not applicable to general $D$ and $\tilde{D}$,
and thus some problems mentioned in \cite[Section~5]{CFS17} remain open.

A major purpose of this paper is to settle down these circumstances.
To be more precise, we shall deduce the Komatu--Loewner equation
for a family of ``continuously'' growing hulls in Section~\ref{sec:KLeq}.
The continuity of growing hulls is introduced in Definition~\ref{dfn:cont}
via the kernel convergence of domains, which is a key concept in this paper.
In Section~\ref{subsec:kernel},
we provide a detailed description on the kernel convergence.
The continuity of hulls and the existence condition \eqref{eq:hull_limit}
of driving function prove to be a complete characteristic
of the Komatu--Loewner evolution in Theorem~\ref{thm:gKLeq}.
Our definition of the continuity is moreover independent of the domain
and conformally invariant, and thus
the chains can be transformed for any domains
(Proposition~\ref{prop:cont} and Theorem~\ref{thm:hcaptrans}).
This systematic treatment of the Komatu--Loewner equation
is our main result.
We further show that our result extends the previous results on
the \emph{locality} of chordal $\skle_{\sqrt{6}, -b_{\bmd}}$ in a full generality,
which solves an open problem in \cite[Section~5]{CFS17}.
Roughly speaking, the locality means that the distribution of
$\skle_{\sqrt{6}, -b_{\bmd}}$ is invariant modulo time-change under conformal maps.
The precise statement is given in Theorem~\ref{cor:locality}.

\section{Preliminaries}
\label{sec:prel}

First of all, let us confirm the usage of basic terms on domains and functions.
\begin{itemize}
\item $\hat{\C} := \C \cup \{\infty\}$ (the Riemann sphere).
\item $B(a, r) := \{ z \in \C; \lvert z-a \rvert < r \}$, $a \in \C$, $r>0$.
\item $\Delta(a, r) := \{ z \in \C; \lvert z-a \rvert > r \}$.
\item $\disk:=B(0, 1)$, $\disk^*:=\Delta(0,1)$.
\item $\Pi$ denotes the mirror reflection
with respect to the real axis $\partial \uhp$.
\item A non-empty set $F \subset \uhp$ is called
a \emph{(compact $\uhp$-)hull} if $F$ is bounded, $F=\uhp \cap \overline{F}$,
and $\uhp \setminus F$ is simply connected.
\end{itemize}
$\{\Delta(0, r) \cup \{\infty\}; r>0\}$ is a fundamental neighborhoods system
of $\infty$ in $\hat{\C}$.
Suppose that $D$ and $\tilde{D}$ are domains in $\hat{\C}$.
A continuous function $f \colon D \to \tilde{D}$ is said to be \emph{univalent}
if it is holomorphic (as a continuous map between two Riemann surfaces)
and injective on $D$.
If further $f$ is surjective, then it is called a \emph{conformal} map.
In other words,
$f$ is conformal if and only if it is a biholomorphic map from $D$ onto $\tilde{D}$.

\subsection{Brownian motion with darning and conformal maps on multiply connected domains}
\label{subsec:BMD_conf}

In this subsection, we summarize the properties of BMD
and some of their applications to the theory of conformal maps.
In particular, Proposition~\ref{prop:2ndorder} will be a key estimate
throughout Section~\ref{subsec:deduction}.

Fix a positive integer $N$ and a simply connected domain $E \subset \C$.
Let $A_j \subset E$, $1 \leq j \leq N$, be mutually disjoint compact continua
such that each $E \setminus A_j$ is connected.
Here, a continuum means a connected closed set consisting of more than one point.
The domain $D := E \setminus \bigcup_j A_j$ is then $N$-connected.
We ``darn'' each hole $A_j$ as follows:
Regarding each $A_j$ as one point $a^*_j$,
we define the quotient topological space $D^*$ by
$D^*:=D \cup \{a^*_1, \ldots, a^*_N\}$.
BMD $(Z^*_t, \mathbb{P}^*_z)$ is defined on $D^*$ by \cite[Definition~2.1]{CFR16}.
The harmonicity for BMD is then defined by \cite[(3.2)]{CFR16}.
The next proposition shows that the BMD-harmonicity
is a stronger condition than the usual harmonicity
for the absorbing Brownian motion~(ABM) on $D$:

\begin{prop}[{\cite{Ch12}} and {\cite[Section~3.3]{CFR16}}]
\label{prop:BMDharmonic}
The following are equivalent for a continuous function $u \colon D^* \to \R$:
\begin{enumerate}
\item \label{cond:BMDharmonic}
$u$ is BMD-harmonic on $D^*$.
\item \label{cond:conjugate}
There is a holomorphic function $f$ on $D$ whose real or imaginary part is $u$;
\end{enumerate}
In particular, a function $u$ on $D$ satisfying Condition~\eqref{cond:conjugate}
extends to a BMD-harmonic function on $D^*$
if it takes a constant limit value on each $\partial A_j$, $1 \leq j \leq N$.
\end{prop}

We define the Green function and Poisson kernel of BMD, like those for ABM.
Let $A_0$ be a hull with piecewise smooth boundary or an empty set,
$E=\uhp \setminus A_0$ and $D$ be as above.
We denote by $G^*_D$ the 0-order resolvent kernel,
or \emph{Green function} of $Z^*$.
Taking the normal derivative, we get the \emph{Poisson kernel} of $Z^*$
\[
K^*_D(z, \xi_0):= -\frac{1}{2}\frac{\partial}{\partial \mathbf{n}_{\xi_0}}G^*_D(z, \xi_0),
\]
where $\mathbf{n}_{\xi_0}$ is the outward unit normal at $\xi_0 \in \partial E$.
The kernels $G^*_D$ and $K^*_D$ can be expressed
by the classical Green function and harmonic measures.
See Sections~4 and 5 in \cite{CFR16} for their concrete expressions.
The following version of Poisson's integral formula holds for the kernel $K^*_D$:
Suppose that a BMD-harmonic function $u$ on $D^*$ vanishes at infinity,
extends continuously to $\partial D^*=\partial E$ and
has a compact support on $\partial E$.
Then by \cite[(5.5)]{CFR16} and the proof of \cite[Theorem~6.4]{CFR16},
$u$ satisfies
\begin{equation} \label{eq:Poi}
u(z)=\mean{*}{z}{u(Z^*_{\zeta^*-})}
=\int_{\partial E} u(\xi_0) K^*_D(z, \xi_0) \,\lvert d\xi_0 \rvert,
\end{equation}
where $\zeta^*$ is the lifetime of $Z^*$.
Note that the former equality holds by the maximum value principle
for BMD-harmonic functions on $D^*$
even if $\partial A_0$ is not smooth.

When $A_j$ is a horizontal slit $C_j$ for each $1 \leq j \leq N$,
we can further define the \emph{complex Poisson kernel} $\Psi_D$ of $Z^*$
by \cite[Lemma~6.1]{CFR16}.
Namely, there is a unique holomorphic function
$\Psi_D(z, \xi_0)$, $\xi_0 \in \partial E$, such that
$\Im \Psi_D(z, \xi_0) = K^*_D(z, \xi_0)$
and $\lim_{z \to \infty} \Psi_D(z, \xi_0) = 0$.
This $\Psi_D$ coincides with $\Psi$ in \cite[Section~2.2]{BF08}
by construction. If $D=\uhp$, namely, no slit is present in $D$,
then the BMD is reduced to the ABM on $\uhp$ and its complex Poisson kernel
$\Psi_{\uhp}$ becomes $-\frac{1}{\pi}\frac{1}{z-\xi_0}$
whose imaginary part is the usual Poisson kernel
$\frac{1}{\pi}\frac{y}{(x-\xi_0)^2+y^2}$ on $\uhp$.
We consider the difference between $\Psi_D$ and $\Psi_{\uhp}$
provided that $E=\uhp$.
In view of the proof of \cite[Lemma~5.6]{CF18}, the function
\begin{equation} \label{eq:PK}
\mathbf{H}_D(z, \xi_0):=\Psi_D(z, \xi_0)+\frac{1}{\pi}\frac{1}{z-\xi_0},
\quad z \in D, \xi_0 \in \partial \uhp,
\end{equation}
can be extended, for each $\xi_0 \in \partial \uhp$, to a holomorphic function
in $z \in D \cup \Pi D \cup \partial \uhp$
after making Schwarz's reflection across $\partial \uhp \setminus \{\xi_0\}$.
The extended function is denoted by $\mathbf{H}_D(z, \xi_0)$ again.
Accordingly, $\Psi_{D}(\cdot, \xi_0)$ extends to a conformal map
from $D \cup \Pi D \cup \partial \uhp \cup \{\infty\}$
onto $\tilde{D} \cup \Pi \tilde{D} \cup \partial \uhp \cup \{\infty\}$.

\begin{rem}
Concerning \eqref{eq:Poi} and the definition of $\Psi_D$,
Lemma~6.1 and Theorem~6.4 of \cite{CFR16} dealt with only the case
where $E=\uhp$.
However, we can easily check that the proof is still valid
for $E=\uhp \setminus A_0$.
Indeed, the BMD complex Poisson kernel for
$E=\uhp \setminus \{z; \lvert z \rvert \leq \varepsilon\}$
appeared in Appendix of \cite{CF18}.
\end{rem}

As an application of BMD to the theory of conformal maps,
\cite[Theorem~7.2]{CFR16} constructed the canonical conformal map
for a hull in $D$ in a probabilistic manner,
which was originally due to \cite[Corollary~3.1]{La06}.
We restate \cite[Theorem~7.2]{CFR16} in the next proposition.

\begin{prop}
\label{prop:canonical}
\begin{enumerate}

\item \label{prop:can_exist}
Let $E=\uhp$ and $D$ be as above.
Suppose that $F$ is a hull contained in $D$ or an empty set.
Then, there exists a unique pair of a standard slit domain $\tilde{D}$ and
conformal map $f_F \colon D \setminus F \to \tilde{D}$
with the hydrodynamic normalization
$\lim_{z \to \infty} (f_F(z)-z)=0$.

\item \label{prop:can_reflect}
The map $f_F$ in \eqref{prop:can_exist} can be extended to a univalent function on
$(D \setminus F) \cup \Pi(D \setminus F)
\cup (\partial \uhp \setminus \overline{F}) \cup \{\infty\}$.
This extended map is denoted by $f_F$ again
and has the following expansion around $\infty$:
\begin{equation} \label{eq:can_expand}
f_F(z)=z + \frac{c}{z} + o(z^{-1}),
\end{equation}
where $c$ is a constant which is positive if $F$ is non-empty.
\end{enumerate}
\end{prop}

\begin{figure}
\centering
\includegraphics[width=12cm]{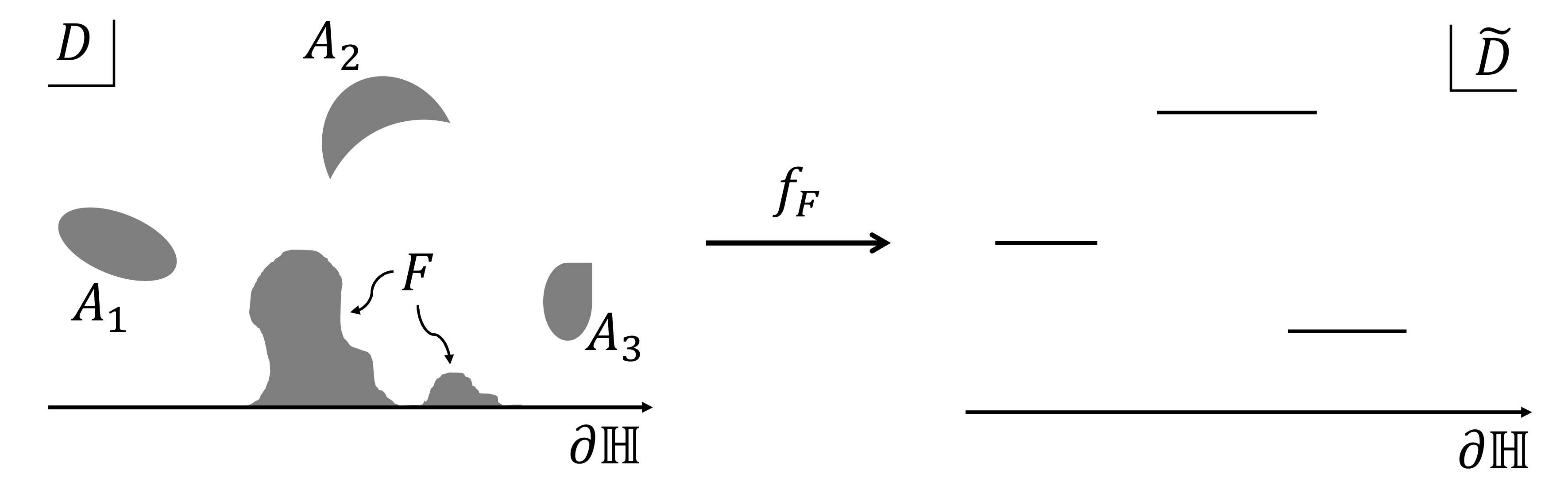}
\caption{Canonical map $f_F$}
\end{figure}

In Proposition~\ref{prop:canonical}, we refer to $f_F$
as the \emph{canonical map from $D \setminus F$ onto $\tilde{D}$}.
The constant $c$ in Proposition~\ref{prop:canonical}~\eqref{prop:can_reflect}
is called the \emph{half-plane capacity of $F$ relative to $D$}
and denoted by $\hcap^D(F)$.
Now a reader familar with the boundary behavior of conformal maps
can skip the following proof and Remark~\ref{rem:canonical},
which are somewhat lengthy due to the exposition on prime ends.

\begin{proof}[Proof of Proposition~\ref{prop:canonical}]
\eqref{prop:can_exist}
The existence of such a pair of standard slit domain $\tilde{D}$ and
conformal map $f_F$ is ensured by \cite[Theorem~7.2]{CFR16} or \cite[Section~2.2]{BF08}.
We prove the uniqueness by the same proof as in \cite[Theorem IX.23]{Ts59},
starting with the summary on the boundary correspondence
induced by conformal maps.

Let $D_0$ be a finitely connected domain.
\begin{itemize}
\item A simple curve $q$ in $\overline{D_0}$ is called a \emph{cross cut}
if both of its end points lie in a single component of $\partial D_0$,
and the other points of $q$ lie in $D_0$.
A cross cut $q$ obviously separates the domain $D_0$ into two components,
that is, $D_0 \setminus q$ consists of two connected components.
\item A sequence $\{q_n\}$ of cross cuts is called a \emph{null-chain}
if all $q_n$ are disjoint, there is a component of $D_0 \setminus q_n$
denoted by $\ins q_n$ such that $\ins q_{n+1} \subset \ins q_n$ for all $n$,
and $\diam q_n \to 0$ as $n \to \infty$.
\item Two null-chains $\{q_n\}$ and $\{q'_n\}$ is said to be \emph{equivalent}
if, for every $m$, there exists a number $n$ such that $\ins q'_n \subset \ins q_m$,
and the same condition with $q_n$ and $q'_n$ exchanged holds.
We call a equivalence class by this relation a \emph{prime end} of $D_0$.
\item $P(D_0)$ denotes the collection of all prime ends of $D_0$.
\end{itemize}
We endow a topology on $D_0 \cup P(D_0)$ as follows:
A subset $U \subset D_0 \cup P(D_0)$ is open if $U \cap D_0$ is open,
and for every prime end $p \in U \cap P(D_0)$, there exists
a null-chain $\{q_n\} \in p$ such that $\ins q_n \subset U \cap D_0$ for some $n$.
Then by definition, a sequence $\{z_m\}$ in $D_0$ converges to a prime end $p$
if and only if, for some null-chain $\{q_n\} \in p$ and each $n$, it holds that
$z_m \in \ins q_{n}$ for sufficiently large $m$.

For the standard slit domain $\tilde{D}$,
the collection of prime ends $P(\tilde{D})$ has a simple expression.
Let $\tilde{C}_j$, $j=1, \ldots, N$, be the slits of $\tilde{D}$
whose left and right end points are $\tilde{z}_j$ and $\tilde{z}^r_j$, respectively.
We use $\partial_p A$, $A \subset \C$, to denote the boundary of $A$
with respect to the path distance topology in $\C \setminus A$.
Then
$\partial_p \tilde{C}_j=\tilde{C}^+_j \cup \tilde{C}^-_j
\cup \{\tilde{z}_j, \tilde{z}^r_j\}$,
where $\tilde{C}^{\pm}_j$ are the upper and lower side of the open slit
$\tilde{C}^0_j:=\tilde{C}_j \setminus \{\tilde{z}_j, \tilde{z}^r_j\}$, respectively.
$P(\tilde{D})$ coincides with the boundary of $\tilde{D}$
in the path distance topology:
$P(\tilde{D})=\partial \uhp \cup \bigcup_j \partial_p \tilde{C}_j$.

It is well known as Carath\'eodory's theorem that, a conformal map
between two finitely connected domains $D_0$ and $D_1$
extends to a homeomorphism between $D_0 \cup P(D_0)$ and $D_1 \cup P(D_1)$.
(See \cite[Theorem~IX.1]{Ts59} or \cite[Theorem~14.3.4]{Co95}.)
Although $D_0$ and $D_1$ are originally supposed to be simply connected
in Carath\'eodory's theorem,
we can easily prove this fact even if the domains have finitely multiple connectivity,
for instance, via the proof of \cite[Theorem~15.3.4]{Co95}.
In our case, the conformal map $f_F \colon D \setminus F \to \tilde{D}$
induces a homeomorphism from
$(D \setminus F) \cup P(D \setminus F)$ onto
$\tilde{D} \cup \partial \uhp \cup \bigcup_j \partial_p \tilde{C}_j$.

Keeping this boundary correspondence in mind, we proceed to
the uniqueness of the pair $(\tilde{D}, f_F)$.
To the contrary, we assume that a pair of a standard slit domain $\tilde{D}_*$ and
conformal map $f_{\ast} \colon D \setminus F \to \tilde{D}_*$
distinct from the pair $(\tilde{D}, f_F)$ enjoys the same property.
The difference $g(z):=f_F(z)-f_*(z)$ is non-constant, holomorphic
on $D \setminus F$ and especially bounded due to the hydrodynamic normalization.
By the above correspondence, the boundary of the image $g(D)$
is written as
\[
\partial g(D)=\{ f_F(z) - f_*(z); z \in P(D \setminus F)\},
\]
which consists of finitely many parallel slits and a subset of $\partial \uhp$.
It is however impossible that such a form of boundary surrounds
a bounded domain $g(D)$, a contradiction.
Thus the uniqueness of the map $f_F$ follows.

\noindent
\eqref{prop:can_reflect}
It is obvious from definition that each point
in $\partial \uhp \setminus \overline{F}$
corresponds to a prime end in $P(D \setminus F)$.
Thus by the boundary correspondence we have
\[
\lim_{z \to \zeta_0, z \in D \setminus F} \Im f_F(z)=0,
\quad \zeta_0 \in \partial \uhp \setminus \overline{F}.
\]
The extension of $f_F$ across $\partial \uhp \setminus \overline{F}$
is now obtained from Schwarz's reflection principle.
The hydrodynamic normalization implies that
$f_F$ has the expansion \eqref{eq:can_expand}.
Finally by \cite[(A.20)]{CF18} we have
\begin{equation} \label{eq:hcap}
c=\frac{2R}{\pi}\int_0^{\pi}
\mean{*}{Re^{i\theta}}{\Im Z^*_{\sigma^*_F}; \sigma_F<\infty} \sin\theta \,d\theta
\end{equation}
for any $R >\sup\{\lvert z \rvert; z \in F \cup \bigcup_j A_j\}$.
Here $\sigma^*_F:=\inf\{t>0; Z^*_t \in F\}$, and $\mathbb{E}^*_z$ denotes
the expectation with respect to $Z^*$ starting at $z \in D^*$.
Although \cite[(A.20)]{CF18} was shown when $D$ is a standard slit domain,
it is also vaild for general $D$ as remarked immediately after \cite[(A.21)]{CF18}.
The expression~\eqref{eq:hcap} implies $c>0$ for a non-empty $F$
since $F$ is non-polar with respect to the ABM on $D$.
\end{proof}

\begin{rem} \label{rem:canonical}
If the closure of the hull $F$ is connected,
which is the case when $F$ is a trace of a simple curve,
then Proposition~\ref{prop:canonical} follows from Theorems~IX.22 and IX.23
of \cite{Ts59} as follows:
Let $D':=(D \setminus F) \cup \Pi (D \setminus F)
\cup (\partial \uhp \setminus \overline{F})$.
There exists a pair of a parallel slit plane $\tilde{D}'$ and a conformal map
$f \colon D' \to \tilde{D}'$ with the normalization
\[
f(z)=z + \frac{c}{z} + o(z^{-1}),\quad z \to \infty,
\]
for some $c \in \C$ with $\Re c>0$ by \cite[Theorem~IX.22]{Ts59}.
Clearly $\hat{f}(z):=\overline{f(\bar{z})}$ satisfies the same condition
with $c$ replaced by $\bar{c}$, and so $f=\hat{f}$ by \cite[Theorem~IX.23]{Ts59}
since $D'$ is of finite connectivity.
Thus $c=\bar{c}>0$, and
$f_F$ is obtained from the restriction of $f$ on $D \setminus F$.

The crutial point is the finite connectivity of $D'$,
which is not necessarily true when $\overline{F}$ is not connected.
Because the uniqueness theorem \cite[Theorem~IX.23]{Ts59} does not work
for the domain of infinite connectivity,
we cannot conclude that $f=\hat{f}$ in the above argument
unless the connectivity of $\overline{F}$ is assumed.
In relation with this remark, we would like to point out that,
in the proof of \cite[Theorem~7.2]{CFR16}, the image of the hull $F$
by the canonical map $\phi \colon \uhp \setminus F \to \uhp$
is stated to be a interval, which is not the case in general.
Needless to say, the proof itself is completely valid since
\cite[Theorem~11.2]{CFR16} used there does not depend on the degree
of connectivity of the domains at issue.
\end{rem}

In the rest of this subsection, $D$ is a standard slit domain
$\uhp \setminus \bigcup_{j=1}^N C_j$.
We denote by $C^{\pm}_j$ the upper and lower side of the slit
$C^0_j:=C_j \setminus \{z_j, z^r_j\}$, respectively,
where $z_j$ and $z^r_j$ are the left and right end points of $C_j$, respectively.
We set $\partial_p C_j := C^+_j \cup C^-_j \cup \{z_j, z^r_j\}$,
which is the boundary of $\uhp \setminus C_j$ in the path distance topology,
as in the proof of Proposition~\ref{prop:canonical}.

Given the canonical map $f_F$ for a hull $F \subset D$,
we can always extend it holomorphically to $\bigcup_j \partial_p C_j$
in the following sense as in \cite[Section~2]{CF18},
which will be used extensively throughout this paper:
Fix $1 \leq j \leq N$ and consider the open rectangles
\begin{gather*}
R_+:=\{z \in \C ;\; x_j < \Re z < x^r_j,\; y_j < \Im z <y_j + \delta\}, \\
R_-:=\{z \in \C ;\; x_j < \Re z < x^r_j,\; y_j - \delta < \Im z <y_j\},
\end{gather*}
and $R:=R_+ \cup C^0_j \cup R_-$,
where $\delta>0$ is taken so small that $R_+ \cup R_- \subset D \setminus F$.
Since $\Im f_F$ takes a constant value on the slit $C_j$
by the boundary correspondence, $f_F$ extends to a holomorphic function $f_F^+$
from $R_+$ to $R$ across $C^0_j$ by Schwarz's reflection.
The extension $f_F^-$ of $f_F|_{R_-}$ across $C^0_j$ is defined in the same way.
As for the extension of $f_F$ on the left end point $z_j$,
we take $\varepsilon>0$ so small that it is less than one-half of the length of $C_j$
and that $B(z_j, \varepsilon) \setminus C_j \subset D \setminus F$.
Then $\psi(z):=(z-z_j)^{1/2}$ maps $B(z_j, \varepsilon) \setminus C_j$
conformally onto $B(0, \sqrt{\varepsilon}) \cap \uhp$, and
$f^l_F:=f_F \circ \psi^{-1}$ extends holomorphically to $B(0, \sqrt{\varepsilon})$
by Schwarz's reflection.
We can also construct the holomorphic function $f^r_F$,
the extension of $f_F$ on the right end point $z^r_j$.
Note that, by the proof of \cite[Lemma~6.1]{CFR16},
the BMD complex Poisson kernel $\Psi_D(z, \xi_0)$
extends holomorphically to $\bigcup_j \partial_p C_j$ 
for each $\xi_0 \in \partial \uhp$ in the same manner.

The canonical map $f_F$ so extended has the following important estimate,
which was originally formulated in \cite[Proposition~6.12]{Dr11} in terms of ERBM:

\begin{prop} \label{prop:2ndorder}
Let $D$ be a standard slit domain.
Suppose that $\xi_0 \in \partial \uhp$ and that
$r_0>0$ satisfies $B(\xi_0, r_0) \cap \uhp \subset D$.
Then for any hull $F \subset D$ with $r:=\inf\{R>0; F \subset B(\xi_0,R)\} \leq r_0$,
it holds that, for all $z \in D \cup \bigcup_j \partial_p C_j$
with $\lvert z-\xi_0 \rvert>r$,
\begin{equation} \label{eq:2ndorder}
\left\lvert z-f_F(z)-\pi\hcap^D(F)\Psi_D(z,\xi_0) \right\rvert \leq C(z)r\hcap^D(F).
\end{equation}
Here $C(z)=C_{D,\xi_0,r_0}(z)>0$ is a locally bounded function depending
only on $D$, $\xi_0$ and $r_0$.
\end{prop}

Proposition~\ref{prop:2ndorder} is a generalization of \cite[Lemma~2.7]{LSW01}
and \cite[Proposition~3.46]{La05} for the upper half-plane
toward the standard slit domain.
Drenning~\cite{Dr11} used it to obtain the Komatu--Loewner equation
for a simple curve in the right derivative sense.
He then discussed the left differentiability by some probabilistic methods
based on the fact that the hull at issue was a simple curve.
In Section~\ref{sec:KLeq}, we also establish the right differentiability
by Proposition~\ref{prop:2ndorder} as he did,
but the subsequent argument is completely different.
We employ the kernel convergence condition instead of his methods
to examine the left differentiability for a family of ``continuously'' growing hulls.

In what follows, we give a complete proof
of Proposition~\ref{prop:2ndorder} by making use of BMD instead of ERBM.
We first quote some estimates on BMD from Appendix of \cite{CF18}.
Let $D$, $r_0$, $F$ and $r$ be as in the assumption
of Proposition~\ref{prop:2ndorder}.
By horizontal translation, we may and do assume $\xi_0=0$
without loss of generality.
Let $D_{\varepsilon}:=D \setminus \overline{B(0,\varepsilon)}$.
By \cite[Proposition~A.2]{CF18}, there is a function $c(z, \theta)$
uniformly bounded in $\lvert z \rvert>r_0$ and $0 \leq \theta \leq \pi$
such that
\begin{equation} \label{eq:BMDest1}
K^*_{D_{\varepsilon}}(z, \varepsilon e^{i\theta})
=2K^*_D(z,0)(1+c(z,\theta)\varepsilon)\sin\theta
\end{equation}
for $\lvert z \rvert>r_0$, $0<\varepsilon<r_0$ and $0 \leq \theta \leq \pi$.
Clearly \eqref{eq:BMDest1} still holds for $z \in \bigcup_j \partial_p C_j$
since $K^*_D(z, \xi_0)$ extends harmonically
as the imaginary part of $\Psi_D(z, \xi_0)$.
By \cite[(A.22) and (A.23)]{CF18},
\begin{equation} \label{eq:BMDest2}
\hcap^D(F)=\frac{2r}{\pi}(1+c'(z,\theta)r)\int_0^{\pi}
\mean{*}{re^{i\theta}}{\Im Z^*_{\sigma^*_F}; \sigma^*_F<\infty}\sin\theta \,d\theta,
\end{equation}
where $c'(z,\theta)$ is a uniformly bounded function in $z$ and $\theta$.

Though it is irrelevant to BMD and rather standard, we remark the following:

\begin{lem}[cf. {\cite[Exercise~2.17]{La05}}] \label{lem:hbound}
Let $n \in \N$ and $u$ be a bounded harmonic function on a domain $V$.
Then, every derivative of $u$ of order $n$ is bounded by
$c(n)\dist(z,\partial V)^{-n}\lVert u \rVert_{\infty}$ for some constant $c(n)$.
\end{lem}

\begin{proof}[Proof of Proposition~\ref{prop:2ndorder}]
Let
\begin{gather*}
h(z):=z-f_F(z)-\pi\hcap^D(F)\Psi_D(z,0), \\
v(z):=\Im h(z)=\Im(z-f_F(z))-\pi\hcap^D(F)K^*_D(z,0).
\end{gather*}
Just as in the proof of \cite[Theorem~A.1]{CF18}, we have
\[
\Im(z-f_F(z)) = \mean{*}{z}{\Im Z^*_{\sigma^*_F}; \sigma^*_F<\infty}.
\]
Denote the right hand side by $v_0(z)$.
From the strong Markov property of $Z^*$, \eqref{eq:Poi},
\eqref{eq:BMDest1} and \eqref{eq:BMDest2}, we obtain,
for $z \in D \cup \bigcup_j \partial_p C_j$ with $\lvert z \rvert > r$,
\begin{align*}
\Im(z-f_F(z)) &= \mean{*}{z}{v_0\left(Z^*_{\sigma^*_{\overline{B(0,r)}}}\right)}
= r\int_0^{\pi} v_0(re^{i\theta})K^*_{D_r}(z, re^{i\theta})\,d\theta \\
&= 2rK^*_D(z,0)\int_0^{\pi}v_0(re^{i\theta})\sin\theta\, d\theta(1+O(r)) \\
&= \pi\hcap^D(F)K^*_D(z,0)(1+O(r)).
\end{align*}
Hence for some $M_1$,
\begin{equation} \label{eq:imbound}
\lvert v(z) \rvert \leq rM_1\hcap^D(F)K^*_D(z,0),
\quad z \in D \cup \bigcup_j \partial_p C_j.
\end{equation}

We now fix $L>r_0$ such that $\bigcup_j C_j \subset B(0,L)$.
Let $\Gamma_z$ be a curve from $iL$ to $z$ in $D \setminus \overline{B(0,r)}$.
Then, $h(z)$ is given by
\begin{equation} \label{eq:lineint}
h(z)=h(iL)+\int_{\Gamma_z}h'(w)\,dw,
\quad z \in D \cup \bigcup_j \partial_p C_j.
\end{equation}
Further by \eqref{eq:imbound}, Lemma \ref{lem:hbound} and
the Cauchy--Riemann equation, we have
\begin{equation} \label{eq:upbound1}
\lvert h'(w) \rvert \leq \frac{M_2\sup_{z' \in \mathcal{N}_{\Gamma_z}} K^*_D(z',0)}
{\dist(\Gamma_z, \partial D)}r\hcap^D(F), \quad w \in \Gamma_z,
\end{equation}
for some constant $M_2$.
$\mathcal{N}_{\Gamma_z}$ is an appropriate neighborhood of $\Gamma_z$.
We describe how to choose it later.
Combining \eqref{eq:upbound1} with \eqref{eq:lineint} yields that
\begin{equation} \label{eq:upbound2}
\lvert h(z) \rvert \leq \lvert h(iL) \rvert
+ \frac{M_2\lvert \Gamma_z \rvert \sup_{z' \in \mathcal{N}_{\Gamma_z}} K^*_D(z',0)}
	{\dist(\Gamma_z, \partial D)}r\hcap^D(F),
\end{equation}
where $\lvert \Gamma_z \rvert$ denotes the length of $\Gamma_z$.

It remains to estimate $\lvert h(iL) \rvert$.
By (A.21) and (A.23) of \cite{CF18},
\[
K^*_D(z,0)=\frac{1}{\pi}\frac{\Im z}{\lvert z \rvert^2}
	+ O\left(\frac{1}{\lvert z \rvert^2}\right),
\]
so that
\begin{equation} \label{eq:Poibound}
K^*_D(z,0) \leq \frac{M_3}{\Im z}, \quad \lvert z \rvert >L.
\end{equation}
Since $v$ is harmonic on $B_y:=\{z \in \C; \lvert z-iy \rvert < y/2\} (\subset D)$
for $y>2L$, it follows from Lemma~\ref{lem:hbound}, \eqref{eq:imbound}
and \eqref{eq:Poibound} that
\[
\lvert v_x(iy) \rvert \leq \frac{2c(1)}{y}\sup_{z \in B_y}\lvert v(z) \rvert
\leq 4c(1)M_1M_3r\hcap^D(F) \cdot \frac{1}{y^2}, \quad y>2L.
\]
Now note that, for $u:=\Re h$, $\lim_{z \to \infty} u(z)=0$
by the properties of $f_F$ and $\Psi_D$.
Consequently by the Cauchy-Riemann equation we have
\begin{equation} \label{eq:rebound}
\lvert u(iL) \rvert \leq \int_L^{\infty}\lvert v_x(iy) \rvert\,dy \leq M_4r\hcap^D(F).
\end{equation}
Here $M_3$ and $M_4$ are constants.

We finally set
\begin{equation} \label{eq:2ndconst}
C(z)=C_{D, 0, r_0}(z)
:=\frac{M_2\lvert \Gamma_z \rvert \sup_{z' \in \mathcal{N}_{\Gamma_z}} K^*_D(z',0)}
{\dist(\Gamma_z, \partial D)}+M_4+M_1K^*_D(iL, 0).
\end{equation}
By choosing appropriate $\Gamma_z$ and $\mathcal{N}_{\Gamma_z}$,
we can take $C$ as a locally bounded function independent of $F$ and $r$.
Thus, \eqref{eq:upbound2}, \eqref{eq:rebound} and \eqref{eq:imbound}
lead us to the desired conclusion.
\end{proof}

Note that \eqref{eq:2ndorder} still holds
with $f_F$ replaced by the extended map $f^+_F$ or other extensions,
since one may define $C(z)$ by taking some appropriate reflection.

\subsection{Initial value problem for the Komatu--Loewner equation}
\label{subsec:initKL}

In this subsection, we describe how one obtains a family of growing hulls
from the initial value problem for the Komatu--Loewner equation.

Fix $N \in \N$ and
let $C_j \subset \uhp$, $1 \leq j \leq N$ be mutually disjoint horizontal slits.
We denote the left and right endpoints of the $j$-th slit $C_j$ by
$z_j=x_j+iy_j$ and $z^r_j = x^r_j+iy_j$, respectively.
Then, the $N$-tuple $(C_j; 1 \leq j \leq N)$ of the slits
are identified with an element
$\slit=(y_1, \ldots, y_N, x_1, \ldots, x_N, x^r_1, \ldots, x^r_N)$ in $\R^{3N}$.
We define the open subset $\Slit$ of $\R^{3N}$ consisting of all such elements by
\begin{align*}
\Slit:=&\{\slit=(y_1, \ldots, y_N, x_1, \ldots, x_N, x^r_1, \ldots, x^r_N) \in \R^{3N} \\
&; y_j > 0, x_j < x^r_j, \ \text{either $x_j < x^r_k$ or $x_k < x^r_j$
whenever $y_j=y_k$, $j \neq k$}\}.
\end{align*}
We denote by $C_j(\slit)$ (resp.\! $D(\slit)$) the $j$-th slit
(resp.\! the standard slit domain) corresponding to $\slit \in \Slit$.
$\Psi_{\slit}:=\Psi_{D(\slit)}$ is the BMD complex Poisson kernel of $D(\slit)$.

For $\xi_0 \in \R$ and $\slit \in \Slit$, we put
\[
b_l(\xi_0, \slit):=\left\{ \begin{split}
&-2\pi \Im\Psi_{\slit}(z_l, \xi_0), &\quad&1 \leq l \leq N, \\
&-2\pi \Re\Psi_{\slit}(z_{l-N}, \xi_0), &\quad&N+1 \leq l \leq 2N, \\
&-2\pi \Re\Psi_{\slit}(z^r_{l-2N}, \xi_0), &\quad&2N+1 \leq l \leq 3N,
\end{split} \right.
\]
where $z_j$ and $z^r_j$, $1 \leq j \leq N$, are the left and right endpoints
of the $j$-th slit $C_j(\slit)$, respectively.
The function $b_l$, $1 \leq l \leq 3N$, has an \emph{invariance
under horizontal translations}, that is,
\[
b_l(\xi_0, \slit) = b_l(0, \slit -\widehat{\xi_0}),
\]
where $\widehat{\xi_0}$ denotes the vector in $\R^{3N}$
whose first $N$ entries are zero and last $2N$ entries are $\xi_0$.
(\cite{CF18} called this property the homogeneity in $x$-direction.)
We can easily check this invariance since
$\Psi_{\slit}(z, \xi_0)=\Psi_{\slit-\widehat{\xi_0}}(z-\xi_0, 0)$.

The Komatu--Loewner equation for slits~\eqref{eq:KLs} is now written as
\begin{equation} \label{eq:KLs2}
\frac{d}{dt}\slit_l(t)=b_l(\xi(t), \slit(t)), \quad 1 \leq l \leq 3N,
\end{equation}
where $\slit_l(t)$ is the $l$-th entry of $\slit(t)$.
Since $b_l$ is locally Lipschitz on $\R \times \Slit$ for each $l$
by \cite[Lemma~4.1]{CF18},
\eqref{eq:KLs2} is solved up to its explosion time $\zeta$.
Here we note that Condition~(L) on a function $f \colon \Slit \to \R$
appearing in \cite[Lemma~4.1]{CF18}
is equivalent to each of the following conditions:
\begin{itemize}
\item the local Lipschitz continuity of $f(\slit)$ in $\slit \in \Slit$,
\item the local Lipschitz continuity of $f(\slit - \widehat{\xi_0})$
in $(\xi_0, \slit) \in \R \times \Slit$.
\end{itemize}
Therefore we simply say that $f$ is locally Lipschitz
if one of these conditions holds.

In this context, we introduce a few more notations.
For a function $f \colon \Slit \to \C$,
we denote $f(\slit - \widehat{\xi_0})$ by $f(\xi_0, \slit)$
and regard it as a function on $\R \times \Slit$
with the invariance under horizontal translations.
Conversely, for a function $\tilde{f} \colon \R \times \Slit \to \C$
with the invariance $\tilde{f}(\xi_0, \slit)=\tilde{f}(0, \slit-\widehat{\xi_0})$,
we denote $\tilde{f}(0, \slit)$ by $\tilde{f}(\slit)$
and regard it as a function on $\Slit$.

Returning to the initial value problem, we set $D_t:=D(\slit(t))$
for the solution $\slit(t)$, $0 \leq t < \zeta$, to \eqref{eq:KLs2}.
The Komatu--Loewner equation~\eqref{eq:KL} is written as
\begin{equation} \label{eq:KL2}
\frac{d}{dt}g_{t}(z) = -2\pi \Psi_{\slit(t)}(g_{t}(z),\xi(t)), \quad g_{0}(z)=z \in D(:=D(\slit)).
\end{equation}
\eqref{eq:KL2} has a unique solution $g_t(z)$ up to
$t_z = \zeta \wedge \sup\{t; \lvert g_t(z) - \xi(t) \rvert >0\}$
by Theorem~5.5~(i) of \cite{CF18}.
By Theorems~5.5, 5.8 and 5.12 of \cite{CF18},
$g_t$ is the canonical map from $D \setminus F_t$ onto $D_t$
where $F_t:=\{z \in D; t_z \leq t\}$, $t \in [0, \zeta)$,
$\{F_t\}$ is a family of growing (i.e.\ strictly increasing) hulls satisfying
\begin{equation} \label{eq:hull_limit}
\bigcap_{\delta>0}\overline{g_t(F_{t+\delta} \setminus F_t)} = \{\xi(t)\}
\end{equation}
for all $t < \zeta$, and further $\hcap^D(F_t)=2t$.
The family $\{F_t\}$, $\{g_t\}$ or $(g_t, F_t)$ here is called
the \emph{Komatu--Loewner evolution driven by $\xi$}.
In the present paper,
we also refer to $\{g_t\}$ as the \emph{Komatu--Loewner chain}.

In the same manner, we introduce
the \emph{stochastic Komatu--Loewner evolution} (SKLE)
as we defined SLE in Section~\ref{sec:intro}.
We say that a function $f \colon \Slit \to \R$ is
\emph{homogeneous with degree $a \in \R$} if, for any $c>0$,
\[
f(c\slit) = c^af(\slit),\quad \slit \in \Slit.
\]
Take two functions $\alpha(\slit)$ and $b(\slit)$ homogeneous
with degree $0$ and $-1$, respectively, and suppose that
both of them satisfy the local Lipschitz condition.
We consider the following SDEs:
\begin{gather}
\xi(t) = \xi + \int_0^t \alpha(\xi(s), \slit(s))\,dB_s
	+ \int_0^t b(\xi(s), \slit(s))\,ds, \label{eq:driv}\\
\slit_l(t)=\slit_l + \int_0^t b_l(\xi(s), \slit(s))\,ds,\quad 1 \leq l \leq 3N, \label{eq:slit}
\end{gather}
where $B_t$ is the one-dimensional standard BM.
The second equation \eqref{eq:slit} is the same as \eqref{eq:KLs2},
though we regard it as a part of the system of SDEs instead of a single ODE.
By the local Lipschitz condition, this system has a unique strong solution
up to its explosion time $\zeta$ (\cite[Theorem~4.2]{CF18}).
The above-mentioned properties also holds for this solution $(\xi(t), \slit(t))$.
We designate the resulting random evolution $\{F_t\}$ as $\skle_{\alpha, b}$.

\section{Convergence of a sequence of univalent functions}
\label{subsec:kernel}

In this section, a version of Carath\'eodory's kernel theorem is formulated,
which is later used to discuss the continuity of growing hulls.
Our discussion seems almost the same as in Chapter~V, Section~5 of \cite{Go69},
but we need some modifications, because Goluzin~\cite{Go69} treated domains
containing $\infty$ (in their interior) while we deal with domains in $\uhp$,
which does not contain $\infty$.
Therefore we provide a detailed description below for the sake of completeness.
The following two facts are fundamental to our argument:

\begin{prop}[Montel] \label{prop:Montel}
A family $\mathcal{H}$ of holomorphic functions on a domain $D \subset \C$
is equicontinuous uniformly on every compact subset of $D$
if it is locally bounded.
In this case $\mathcal{H}$ is a normal family on $D$, i.e.,
relatively compact in the topology of locally uniform convergence.
\end{prop}

\begin{prop} \label{prop:univlim}
If a sequence $\{f_n\}$ of univalent functions on a domain $D$ converges
to a non-constant function $f$ uniformly on every compact subset of $D$,
then $f$ is also univalent on $D$.
\end{prop}

In addition, the following two classes of univalent functions are significant:
First, we define the set $S$ as the totality of univalent functions
$f \colon \disk \to \C$ satisfying $f(0)=0$ and $f'(0)=1$.
In other words, a univalent function $f \colon \disk \to \C$ belongs to $S$
if and only if $f(z)$ has the following power series expansion around the origin:
\begin{equation} \label{eq:classS}
f(z)=z+\sum_{n=2}^{\infty}a_n z^n.
\end{equation}
Next, we define the set $\Sigma$ as the totality of univalent functions
$f \colon \disk^* \to \C$ satisfying $f(\infty)=\infty$
and $\res(f, \infty)=1$.
In other words, a univalent function $f \colon \disk^* \to \C$
belongs to $\Sigma$ if and only if $f(z)$ has the following Laurent series expansion
around $\infty$:
\begin{equation} \label{eq:classSigma}
f(z)=z+b_0+\sum_{n=1}^{\infty}b_n z^{-n}.
\end{equation}

\begin{prop}[Area theorem, Gronwall] \label{prop:areathm}
Suppose $f \in \Sigma$ with Laurent series expansion \eqref{eq:classSigma}.
Then it holds that $\sum_{n=1}^{\infty}n\lvert b_n \rvert^2 \leq 1$.
\end{prop}

\begin{proof}
See for instance \cite[Theorem~II.4.1]{Go69}
or Inequality (5) in \cite[Section~1.2]{Po75}.
\end{proof}

\begin{prop}[Bieberbach] \label{prop:Bie}
Suppose $f \in S$ with power series expansion \eqref{eq:classS}.
Then it holds that $\lvert a_2 \rvert \leq 2$.
\end{prop}

\begin{proof}
See Chapter~II, Section~4 of \cite{Go69} or \cite[Theorem~1.5]{Po75}.
\end{proof}

\begin{lem}[{\cite[Theorem~II.4.3 and Lemma~V.2.2]{Go69}}] \label{lem:Sigma}
Suppose $f \in \Sigma$ with Laurent series expansion \eqref{eq:classSigma}.
Then $\C \setminus f(\disk^*) \subset \overline{B(b_0, 2)}$
and $\lvert f(z)-b_0 \rvert \leq 2\lvert z \rvert$ for $z \in \disk^*$.
\end{lem}

\begin{proof}
For any $c \in \C \setminus f(\disk^*)$, the function
\[
f_c(z):=\frac{1}{f(1/z)-c}=z+(c-b_0)z^2+\cdots
\]
belongs to $S$, and by Bieberbach's theorem \ref{prop:Bie} we have
$\lvert c-b_0 \rvert \leq 2$.
Hence the former part of the lemma follows.

To prove the latter, we consider the function
\[
F_w(z):=\frac{1}{w}f(wz)-\frac{b_0}{w}
=z + \frac{b_1}{w^2z} + \cdots
\]
for $w \in \disk^*$. Since $F_w \in \Sigma$,
we have $\partial F_w(\disk^*) \subset \overline{B(0, 2)}$
by the former part of the lemma.
In particular $\lvert F_w(1) \rvert \leq 2$, that is,
$\lvert f(w)-b_0 \rvert \leq 2\lvert w \rvert$.
\end{proof}

The following corollary easily follows from Lemma~\ref{lem:Sigma}:

\begin{cor} \label{cor:precptSigma}
Suppose that $D \subset \C$ is a domain containing $\Delta(0, r)$
for some $r>0$ and that $\mathcal{H}=\{f_{\lambda}; \lambda \in \Lambda\}$
is a family of univalent functions on $D$ with Laurent series expansion
\[
f_{\lambda}(z)=z+b^{(\lambda)}_0+\sum_{n=1}^{\infty}b^{(\lambda)}_n z^{-n}
\]
around $\infty$.
Then $\mathcal{H}$ is locally bounded on $D$
if and only if $\{b^{(\lambda)}_0; \lambda \in \Lambda\}$ is bounded.
\end{cor}

We now turn to the definition of kernel,
a key notion throughout our discussion in Section~\ref{sec:KLeq}.
To clarify the role of each hypothesis
in the kernel theorem~\ref{prop:kerthmH} below,
we mention our hypotheses in a fashion slightly more abstract
than we need in this paper.
Let $\{D_n; n \in \N\}$ be a sequence of domains in $\uhp$.
We assume that
\begin{enumerate}
\item[(K.1)] there exists a constant $L>0$ such that
$\Delta(0, L) \cap \uhp \subset D_n$ for all $n$.
\end{enumerate}

\begin{dfn} \label{def:kernel}
Under Assumption (K.1), the \emph{kernel} of $\{D_n\}$ is defined
as the largest unbounded domain $D$ such that
each compact subset $K \subset D$ is included by $\bigcap_{n \geq n_K}D_n$
for some $n_K \in \N$.
If every subsequence of $\{D_n\}$ has the same kernel,
then we say that \emph{$\{D_n\}$ converges to $D$
in the sense of kernel convergence} and denote it simply by $D_n \to D$.
\end{dfn}

In other words, the kernel $D$ is an unbounded connected component of the set
of all points $z$ such that $B(z, r_z) \subset D_n$, $n \geq n_z$, for some $r_z>0$
and $n_z \in \N$. By Assumption (K.1),
$D$ always exists, is unique and contains $\Delta(0, L) \cap \uhp$.

Let $D$ be the kernel of $\{D_n\}$ and $f$ and $f_n$, $n \in \N$, be functions on $D$ and $D_n$, respectively.
If $\{f_n\}$ converges to $f$ uniformly on each compact subset $K$ of $D$,
then we say as usual that \emph{$\{f_n\}$ converges to $f$
uniformly on compacta} and denote it by $f_n \to f$ u.c.\ on $D$.
This convergence makes sense since $K$ is included by $D_n$
for sufficiently large $n$.
In what follows, we assume that each $f_n \colon D_n \to \C$ is univalent
and enjoys the following two conditions:
\begin{enumerate}
\item[(K.2)]
$\lim_{z \to \infty}(f_n(z)-z)=0$;
\item[(K.3)]
$\lim_{z \to \xi_0, z \in D_n} \Im f_n(z)=0$
for all $\xi_0 \in \partial \uhp \cap \Delta(0, L)$;
\end{enumerate}
where $L$ is the constant in Assumption (K.1).
Note that, as is easily seen, Propositions~\ref{prop:Montel} and \ref{prop:univlim}
and Corollary~\ref{cor:precptSigma} hold even for the moving domains $D_n$.
By Schwarz's reflection, Assumption (K.3) means that $f_n$ can be extended
to a univalent function on the domain $D_n \cup \Delta(0, L)$.
We denote the extended map by $f_n$ again.
Then by Assumption (K.2),
$f_n$ has the Laurent expansion
$f_n(z)=z+a_n/z+o(z^{-1})$ around $\infty$ for some constant $a_n$,
and $\{f_n\}$ is a normal family on $D$ by Corollary~\ref{cor:precptSigma}
and Montel's theorem~\ref{prop:Montel}.

Under (K.1)--(K.3) and some additional assumptions,
we prove a version of the kernel theorem,
which relates the u.c.\ convergence of $\{f_n\}$
to the kernel convergence of $\{f_n(D_n)\}$,
mainly following the proof of \cite[Theorem~V.5.1]{Go69}.

\begin{thm}[Kernel theorem] \label{prop:kerthmH}
Suppose that $\{D_n\}$ and $\{f_n\}$ satisfy Assumptions {\rm (K.1)--(K.3)}
and that there exist mutually disjoint subsets $A_0$, $A_1$, ..., $A_N$, $N \in \N$,
of $\uhp$ with the following conditions:
\begin{itemize}
\item $A_0$ is a hull or an empty set;
\item Each $A_j$, $1 \leq j \leq N$, is a compact, connected set
with $\uhp \setminus A_j$ connected;
\item $D_n \to D:=\uhp \setminus \bigcup_{j=0}^N A_j$.
\end{itemize}
Let $\tilde{D}_n:=f_n(D_n)$.
Then the following are equivalent:
\begin{enumerate}
\item \label{cond:ucconv} There exists a univalent function $f$ on $D$
such that $f_n \to f$ u.c.\ on $D$;
\item \label{cond:kerconv} There exists a domain $\tilde{D}$
such that $\tilde{D}_n \to \tilde{D}$.
\end{enumerate}
If one of these conditions happens, then $\tilde{D}=f(D)$, and
$f_n^{-1} \to f^{-1}$ u.c. on $\tilde{D}$.
\end{thm}

\begin{lem} \label{lem:K1forIm}
Under Assumptions {\rm (K.1)--(K.3)}, the sequence $\{\tilde{D}_n\}$
in Theorem~\ref{prop:kerthmH} enjoys Condition {\rm (K.1)}
with the constant $L$ in {\rm (K.1)} replaced by $2L$.
\end{lem}

\begin{proof}
Since $L^{-1}f_n(Lz) \in \Sigma$ by (K.1)--(K.3),
we have $\C \setminus L^{-1}f_n(L\disk^*) \subset \overline{B(0, 2)}$
by Lemma~\ref{lem:Sigma}, that is, $f_n(\Delta(0, L)) \supset \Delta(0, 2L)$.
\end{proof}

By Lemma~\ref{lem:K1forIm}, we can define the kernel of $\{\tilde{D}_n\}$,
which is denoted by $\tilde{D}$.

\begin{proof}[Proof of Theorem~\ref{prop:kerthmH}]
\noindent
$\eqref{cond:ucconv} \Rightarrow \eqref{cond:kerconv}$:
Assume \eqref{cond:ucconv}.
Note that $f$ is univalent on $D$ by Proposition~\ref{prop:univlim}.
What we should prove is that any subsequence of $\{\tilde{D}_n\}$
has the same kernel $\tilde{D}$.

We first show that $f(D) \subset \tilde{D}$.
Fix an arbitrary compact subset $K$ of $f(D)$.
We take a bounded domain $V$ with smooth boundary so that
$K \subset V \subset \overline{V} \subset f(D)$ and
put $\delta:=\dist(K, \partial V)/2>0$.
We then have $\lvert f(z)-w \rvert > \delta$ for $w \in K$ and
$z \in \partial f^{-1}(V)$.
On the other hand, there is some $n_{K, V} \in \N$ such that
$\lvert f_n(z)-f(z) \rvert < \delta$ for $z \in \overline{f^{-1}(V)}$ and
$n \geq n_{K, V}$, since $\{f_n\}$ converges to $f$ uniformly
on the compact subset $\overline{f^{-1}(V)}(=f^{-1}(\overline{V}) \subset D)$.
Thus by the equation
\[
f_n(z)-w = (f_n(z)-f(z)) + (f(z)-w),
\]
we can conclude from Rouch\'e's theorem that all the functions $f_n(z)-w$
for $w \in K$ and $n \geq n_{K, V}$ have exactly one zero in $f^{-1}(V)$.
This implies that $K \subset f_n(f^{-1}(V)) \subset \tilde{D}_n$
for $n \geq n_{K, V}$, and so $f(D) \subset \tilde{D}$ by definition.

We next consider the inverse map $f_n^{-1}$.
By the Laurent expansion of $f_n$ and Lemma~\ref{lem:K1forIm},
$f_n^{-1}$ also has the expansion $f_n^{-1}(z)=z-a_n/z+\cdots$, $z \to \infty$.
By Corollary~\ref{cor:precptSigma} and Montel's theorem~\ref{prop:Montel},
$\{f_n^{-1}\}$ is a normal family and so
has a subsequence $\{f_{n_k}^{-1}\}$ converging u.c.\ on $\tilde{D}$.
We can check that the limiting univalent function $g:=\lim_{k \to \infty}f_{n_k}^{-1}$
is the inverse map of $f$ on $f(D)$ as follows:
For a fixed $z \in D$, we take $N$ so large that
$\{f_{n_k}(z); k \geq N\} \cup \{f(z)\}$ is a bounded subset of $f(D)$.
Since $\{f_{n_k}^{-1}\}_k$ converges and is equicontinuous
uniformly on this compact set, we have
\begin{align*}
g(f(z))-z&=\{g(f(z))-f_{n_k}^{-1}(f(z))\}+\{f_{n_k}^{-1}(f(z))-f_{n_k}^{-1}(f_{n_k}(z))\} \\
&\to 0 \quad \text{as} \quad k \to \infty.
\end{align*}
Hence $g|_{f(D)}=f^{-1}$, independent of the choice
of the subsequence $\{f_{n_k}^{-1}\}$.
By the identity theorem, any convergent subsequence of $\{f_n^{-1}\}$
has the same limit $g$ on the whole $\tilde{D}$.
Thus the original sequence $\{f_n^{-1}\}$ converges to $g$ u.c.\ on $\tilde{D}$.

Reversing the roles of $f_n$ and $f_n^{-1}$ at the beginning of this proof,
we have $g(\tilde{D}) \subset D$.
In particular, since $g|_{f(D)}=f^{-1}$ and $f(D) \subset \tilde{D}$, it follows that
$D=g(f(D)) \subset g(\tilde{D}) \subset D$, which yields $f(D)=\tilde{D}$.
If we repeat the argument so far for any subsequence of $\{\tilde{D}_n\}$,
then we see that it has the same kernel $f(D)$.
Hence $\tilde{D}_n \to f(D)$, which completes the proof of
$\eqref{cond:ucconv} \Rightarrow \eqref{cond:kerconv}$.
Note that this proof also establishes the latter part of the proposition,
that is, $f_n^{-1} \to f^{-1}$ u.c.\ on $\tilde{D}=f(D)$.

\noindent
$\eqref{cond:kerconv} \Rightarrow \eqref{cond:ucconv}$:
Assume \eqref{cond:kerconv}.
Contrary to our claim, we suppose that \eqref{cond:ucconv} is false.
Since $\{f_n\}$ is a normal family on $D$,
there are at least two subsequences $\{f^{(1)}_n\}_n$ and $\{f^{(2)}_n\}_n$
of $\{f_n\}$ converging to distinct limits $f^{(1)}$ and $f^{(2)}$, respectively,
u.c.\ on $D$.
By the implication~$\eqref{cond:ucconv} \Rightarrow \eqref{cond:kerconv}$
already proven, $\{f^{(k)}_n(D_n)\}_n$, $k=1, 2$, converge in the sense of
kernel convergence, and their limits are the same domain $\tilde{D}$
by our hypothesis~\eqref{cond:kerconv}.
Then the composite ${f^{(1)}}^{-1} \circ f^{(2)}$ is a conformal automorphism on $D$
that is not the identity map $\mathrm{id}_D$.
If $A_1$, ..., $A_N$ are all continua, then we can take the canonical map
$g \colon D \to \hat{D}$, where $\hat{D}$ is a standard slit domain.
In this case, $g \circ {f^{(1)}}^{-1} \circ f^{(2)}$ is also the canonical map on $D$,
and by the uniqueness of canonical map we have
${f^{(1)}}^{-1} \circ f^{(2)}=\mathrm{id}_D$, a contradiction.
If some of $A_k$, say $A_{l+1}$, $A_{l+2}$, ..., $A_N$, are singletons,
then we can easily see, as in \cite[Exercise~15.2.1]{Co95},
that ${f^{(1)}}^{-1} \circ f^{(2)}$ is extended to a conformal automorphism
on $D':=\uhp \setminus \bigcup_{j=0}^{l} A_j$ since the singularities
$A_{l+1}$, $A_{l+2}$, ..., $A_N$ are removable.
Hence ${f^{(1)}}^{-1} \circ f^{(2)}=\mathrm{id}_{D'}$ by the same argument.
Thus in any case we arrive at a contradiction, which yields \eqref{cond:ucconv}.
\end{proof}

Note that, in the proof
of the implication~$\eqref{cond:ucconv} \Rightarrow \eqref{cond:kerconv}$ above,
we do not use the hypothesis that the kernel $D$ of $\{D_n\}$ has the form
$\uhp \setminus \bigcup_{j} A_j$.
We need this hypothesis only for proving the uniqueness of automorphism on $D$.
In Goluzin's proof of \cite[Theorem~V.5.1]{Go69}, this uniqueness follows
from the property of $\Sigma$ applied to ${f^{(1)}}^{-1} \circ f^{(2)}$,
but in our case, $f^{(1)}$ and $f^{(2)}$ extend only to $\Delta(0, L) \cup D$,
not enough to mimic his argument.
This is the reason why we have to suppose that $D=\uhp \setminus \bigcup_{j} A_j$.

\section{Komatu--Loewner equation for a family of growing hulls}
\label{sec:KLeq}

\subsection{Deduction of the Komatu--Loewner equation}
\label{subsec:deduction}

In this subsection, we define the continuity of a family of growing hulls
and deduce the Komatu--Loewner equation for such hulls.

Here is our basic setting throughout this subsection.
Let $\{F_t; 0 \leq t < t_0\}$ be a family of growing hulls
(i.e.\ strictly increasing hulls) in a fixed standard slit domain $D$.
For each $t$, let $g_t \colon D \setminus F_t \to D_t$ be the canonical map,
$a_t:=\hcap^D(F_t)$ and $\slit(t) \in \Slit$ correspond to the slits $\{C_j(t)\}$
of $D_t$.
$C_j(t)$ is sometimes denoted by $C_{j,t}$ as well.
We further define, for $0 \leq s \leq t < t_0$,
\[
g_{t,s}:=g_s \circ g_t^{-1} \colon D_t \to D_s \setminus g_s(F_t \setminus F_s).
\]
Clearly $g_s(F_t \setminus F_s)$ is a hull, and
$g_{t,s}^{-1}$ is the canonical map on $D_s \setminus g_s(F_t \setminus F_s)$.
Moreover for a fixed $t_1 \in (0, t_0)$,
the family $\{(D \setminus F_t, g_t); t \in [0, t_1]\}$
satisfies Assumptions (K.1)--(K.3) in Section~\ref{subsec:kernel}.
Indeed, the constant $L=L_{t_1}$ in (K.1) can be taken so that
$F_{t_1} \cup \bigcup_j C_j \subset B(0, L_{t_1})$.
(K.2) and (K.3) are obvious.
Thus we can apply the theory developed in Section~\ref{subsec:kernel}
to $g_t$ and $g_{t, s}$ over each compact subinterval of $[0, t_0)$.

\begin{figure}
\centering
\includegraphics[width=13cm]{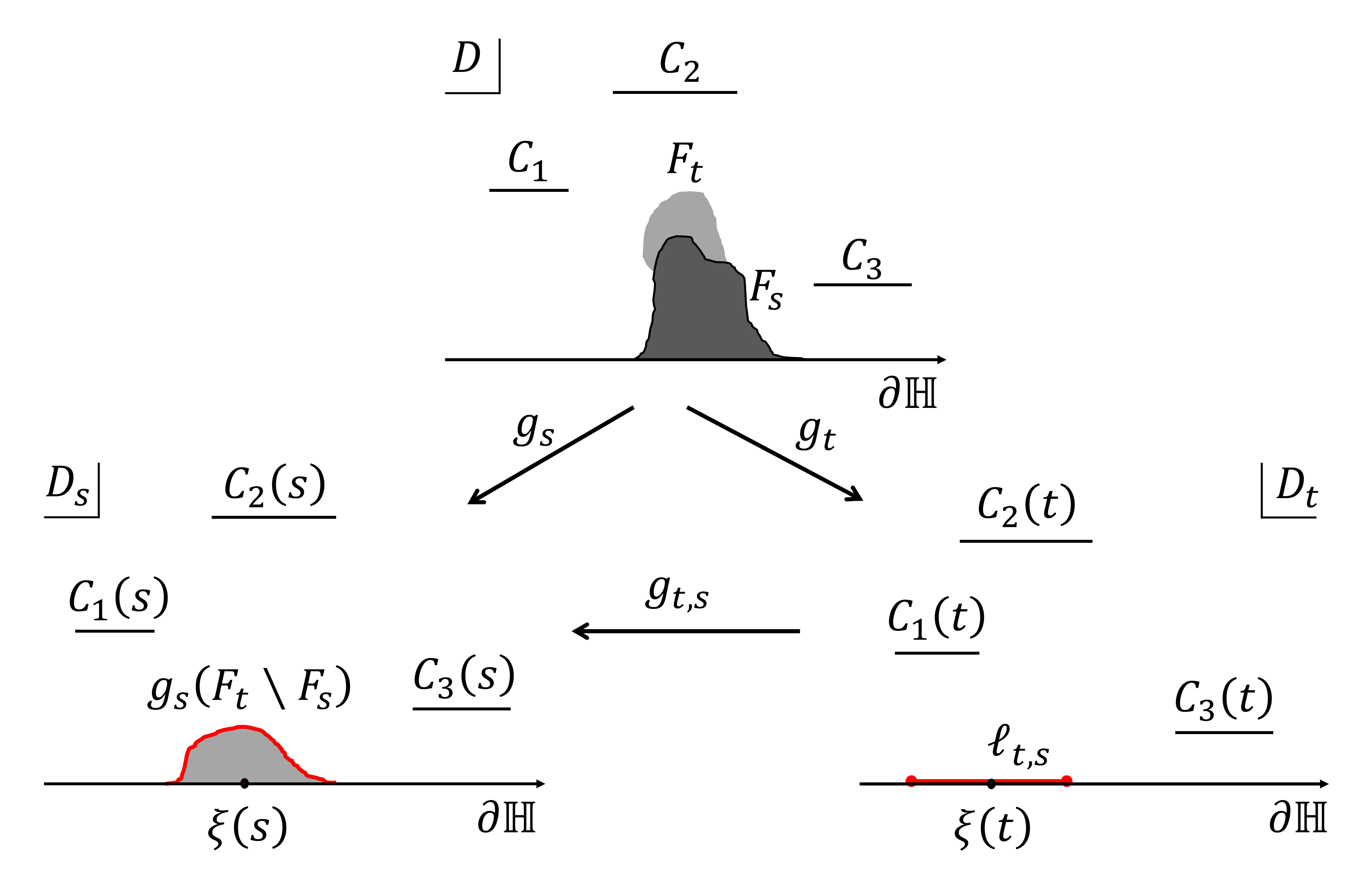}
\caption{Conformal map $g_{t,s}$}
\end{figure}

In what follows, several conditions are imposed on $\{F_t\}$.
If there exists a function $\xi \colon [0, t_0) \to \R$ such that
\eqref{eq:hull_limit} holds for any $t \in [0, t_0)$,
then we call $\xi$ the \emph{driving function of $\{F_t\}$}.
The condition \eqref{eq:hull_limit} is sometimes called the right continuity of $\{F_t\}$
and employed in the existing literature, for example,
\cite[Section~4.1]{La05}, \cite[Section~4]{La06} and \cite[Section~6]{CF18}.
One reason is that, for a family of growing hulls having this property,
we can obtain the Komatu--Loewner equation
in the right derivative sense as in Proposition~\ref{prop:KLright}.
However, it should be noted that we mean a weaker condition
than \eqref{eq:hull_limit} by the ``right continuity'' in Definition \ref{dfn:cont}.

\begin{prop} \label{prop:KLright}
Let $\{F_t\}_{t \in [0,t_0)}$ be a family of growing hulls in $D$
with driving function $\xi \colon [0, t_0) \to \R$.
\begin{enumerate}
\item \label{prop:hcapright}
The half-plane capacity $a_t$ is strictly increasing and right continuous in $t$.
\item \label{prop:ucright}
$g_{t,s}^{-1}(z) \to z$ u.c.\ on $D_s$ as $t \downarrow s$ for any $s \in [0, t_0)$.
\item \label{prop:rdiff}
$g_t(z)$ is right differentiable in $a_t$
for each $z \in D \cup \bigcup_j \partial_p C_j$, and
\begin{equation} \label{eq:KLright}
\frac{\partial^+ g_t(z)}{\partial a_t} = -\pi \Psi_{D_t}(g_t(z), \xi(t)),
\quad g_0(z)=z,\quad t \in [0, t_0).
\end{equation}
Here $\partial^+ g_t(z)/\partial a_t$ denotes the right derivative of $g_t(z)$
with respect to $a_t$.
\end{enumerate}
\end{prop}

\begin{proof}
\eqref{prop:hcapright}
Let $0 \leq s \leq t < t_0$.
We can easily observe that
\[
a_t-a_s=\hcap^{D_s}(g_s(F_t \setminus F_s)).
\]
Since $g_s(F_t \setminus F_s)$ is non-polar (with respect to the ABM on $D_s$)
for $t>s$, the right hand side is positive by \eqref{eq:hcap}.
\eqref{eq:hcap} also implies $\lim_{t \downarrow s}(a_t-a_s)=0$
because $\sup\{\Im z; z \in g_s(F_t \setminus F_s)\} \to 0$ as $t \downarrow s$.

\eqref{prop:ucright} and \eqref{prop:rdiff} are immediate consequences
of \eqref{prop:hcapright} and Proposition~\ref{prop:2ndorder}.
\end{proof}

The left continuity of $a_t$ and left differentiability of $g_t(z)$ do not follow
from \eqref{eq:hull_limit}.
To proceed further, we define the continuity of $\{F_t\}$
as the continuity of $D \setminus F_t$ in the sense of kernel convergence.

\begin{dfn} \label{dfn:cont}
$\{F_t\}_{t \in [0,t_0)}$ is said to be
\emph{(left/right) continuous in $D$ at $t \in [0, t_0)$}
if $D \setminus F_u \to D \setminus F_t$
as $u$ approaches $t$ (from left/right).
\end{dfn}

Such a continuity condition did not appear
in the recent studies \cite{BF08, La06, Dr11, CFR16, CF18},
but it is not new in complex analysis.
Indeed, a similar condition was imposed
when Pommerenke established a version of the radial Loewner equation
in \cite[Section~6.1]{Po75}.
Below we show that Definition~\ref{dfn:cont} works well
even when the domain has multiple connectivity.

\begin{lem} \label{lem:sff_rcont}
If $\{F_t\}$ satisfies \eqref{eq:hull_limit} for some $\xi(s) \in \R$ at $s \in [0, t_0)$,
then it is right continuous in $D$ at $s$.
\end{lem}

\begin{proof}
By \eqref{eq:hull_limit} and Proposition~\ref{prop:KLright}~\eqref{prop:ucright},
we get the following two convergences as $t \downarrow s$:
\[
D_s \setminus g_s(F_t \setminus F_s) \to D_s,\quad
g_{t,s}^{-1}(z) \to z \ \text{u.c.}
\]
Hence $D_t \to D_s$ and $g_{t,s}(z) \to z$ u.c.\ 
as $t \downarrow s$ by the kernel theorem~\ref{prop:kerthmH}.
Since $g_t^{-1}=g_{t,s} \circ g_s^{-1}$,
it also holds that $g_t^{-1} \to g_s^{-1}$ u.c.\ as $t \downarrow s$.
By the kernel theorem~\ref{prop:kerthmH} again,
we obtain $D \setminus F_t \to D \setminus F_s$.
\end{proof}

\begin{lem} \label{lem:cont}
\begin{enumerate}
\item \label{lem:lcont}
Suppose that $\{F_s\}$ is left continuous in $D$ at $t \in (0, t_0)$.
Then $D_s \to D_t$ as $s \uparrow t$, that is,
$\slit(s)$ is left continuous at $t$.
Moreover, $g_{t,s}^{-1}(z) \to z$ u.c.\ on $D_t$ as $s \uparrow t$,
and $a_s$ is left continuous at $t$.
\item \label{lem:rcont}
Suppose that $\{F_t\}$ is right continuous in $D$ at $s \in [0, t_0)$.
Then $D_t \to D_s$ as $t \downarrow s$, that is,
$\slit(t)$ is right continuous at $s$.
Moreover, $g_{t,s}^{-1}(z) \to z$ u.c.\ on $D_s$ as $t \downarrow s$,
and $a_t$ is right continuous at $s$.
\end{enumerate}
\end{lem}

\begin{proof}
We prove only \eqref{lem:lcont}
because the proof of \eqref{lem:rcont} is quite similar.

Since the family $\{(D \setminus F_s, g_s); s \in [0, t]\}$
satisfies (K.1)--(K.3),
we have $D_s \supset \Delta(0, 2L_t) \cap \uhp$ by Lemma~\ref{lem:K1forIm}.
This implies $\bigcup_{s \in [0,t]} \bigcup_j C_{j,s} \subset \overline{B(0,2L_t)}$,
that is, $\{\slit(s); s \in [0, t]\}$ is bounded.
We can thus take a sequence $\{s_n\}$ with $s_n \uparrow t$ so that
$\widetilde{\slit}:=\lim_{n \to \infty}\slit(s_n)$ exists in $\R^{3N}$.
Though $\widetilde{\slit}$ is not necessary in $\Slit$,
it is obvious from definition that $D_{s_n}$ converges to a slit domain $\tilde{D}$
in the sense of kernel convergence.
Some of the slits of $\tilde{D}$ may degenerate.
Since $D \setminus F_{s_n} \to D \setminus F_t$ by the left continuity of $\{F_s\}$,
we can apply the kernel theorem~\ref{prop:kerthmH} to $\{g_{s_n}\}$
to obtain the limiting conformal map
$\tilde{g}:=\lim_{n \to \infty}g_{s_n} \colon D \setminus F_t \to \tilde{D}$.
Then, all the slits of $\tilde{D}$ must not degenerate, and
$\tilde{g}$ must be the canonical map on $D \setminus F_t$,
which yields $\tilde{g} = g_t$ and $\tilde{D} = D_t$
by the uniqueness in Proposition~\ref{prop:canonical}.
In particular, this limit is independent of the choice of $\{s_n\}$.
We therefore conclude that $D_s \to D_t$ as $s \uparrow t$.

The equivalence between the left continuity of $\{D_s\}$ and that of $\slit(s)$
can be checked easily from definition, and so we omit it.

Since $D \setminus F_s \to D \setminus F_t$ and $D_s \to D_t$ as $s \uparrow t$,
the kernel theorem~\ref{prop:kerthmH} implies $g_s^{-1} \to g_t^{-1}$ u.c.,
which in turn yields $g_{t,s}(z)^{-1} \to z$ u.c. as $s \uparrow t$.
To show the left continuity of $a_s$, we regard $h_s(z):=(2L_t)^{-1}g_{t,s}^{-1}(2L_tz)$
as an element of $\Sigma$ by Schwarz's reflection.
Writing the Laurent series expansion around infinity as
\[
h_s(z)=z+\frac{a_t-a_s}{4L_t^2}\frac{1}{z}+\sum_{n \geq 2}\frac{c_{n,s}}{z^n},
\]
we get, from the Cauchy--Schwarz inequality
and the area theorem~\ref{prop:areathm},
\begin{align*}
&\left\lvert h_s(z)-z \right\rvert
= \left\lvert \frac{a_t-a_s}{4L_t^2}\frac{1}{z}
	+ \sum_{n \geq 2}\frac{c_{n,s}}{z^n} \right\rvert
\geq \left\lvert \frac{a_t-a_s}{4L_t^2}\frac{1}{z} \right\rvert
	- \left\lvert \sum_{n \geq 2}\frac{c_{n,s}}{z^n} \right\rvert \\
&\geq \frac{a_t-a_s}{4L_t^2}\frac{1}{\lvert z \rvert}
	- \left(\sum_{n \geq 2}\lvert c_{n,s} \rvert^2\right)^{1/2}
	\left(\sum_{n \geq 2}\lvert z \rvert^{-2n}\right)^{1/2}
\geq \frac{a_t-a_s}{4L_t^2}\frac{1}{\lvert z \rvert}
	- \frac{\lvert z \rvert^{-4}}{1-\lvert z \rvert^{-2}}.
\end{align*}
Since $\lim_{s \uparrow t}\lvert h_s(z)-z \rvert = 0$ for any $z \in D_t$,
we have $\lim_{s \uparrow t}(a_t - a_s)=0$.
\end{proof}

By Lemma~\ref{lem:cont}, $a_t$ is a strictly increasing continuous function
on $[0, t_0)$ if $\{F_t\}$ is continuous.
We can thus reparametrize $\{F_t\}$ so that $a_t$ is differentiable in $t$.
As a particular case, we say that $\{F_t\}$ obeys
the \emph{half-plane capacity parametrization in $D$}
if $a_t=\hcap^D(F_t)=2t$.

\begin{lem} \label{lem:jointcont}
Suppose that $\{F_t\}$ is continuous in $D$ at every $t \in [0, t_0)$.
Then $g_t(z)$ is jointly continuous in
$(t,z) \in [0, t_0) \times (D \cup \bigcup_j \partial_p C_j)$.
\end{lem}

\begin{proof}
$g_t(z)$ is jointly continuous on $[0, t_0) \times D$
since $g_s \to g_t$ u.c.\ on $D$ as $s \to t$ for any $t \in [0, t_0)$
by Lemma~\ref{lem:cont}.
Recall from Section~\ref{subsec:BMD_conf} that the canonical map $g_t=f_{F_t}$
can be extended holomorphically to $\bigcup_j \partial_p C_j$.
For a fixed $j$, let $g^+_t$ be the extended map of $g_t$ from
\[
R_+=\{z \in \C ;\; x_j < \Re z < x^r_j,\; y_j < \Im z <y_j + \delta\}
\]
to $R=R_+ \cup C_j \cup R_-$ across $C^0_j$.
Here we use the notation in Section~\ref{subsec:BMD_conf}.
For a fixed $t_1\in (0, t_0)$, $\{g^+_t\}_{t \in [0, t_1]}$ is locally bounded
on $R_+ \cup R_-$ by the local boundedness of $\{g_t\}$,
and also on $C^0_j$ since $g^+_t(C^0_j) \subset C_{j,t} \subset B(0, 2L_{t_1})$.
Hence $\{g^+_t\}_{t \in [0, t_1]}$ is a normal family on $R$.
Any sequence $\{g^+_{s_n}\}$, $s_n \to t \in [0, t_1]$,
converging u.c.\ when $n \to \infty$ has the same limit,
because it converges to $g_t$ on $R_+$ and so the identity theorem applies
to $\lim_{k \to \infty}g^+_{s_n}$.
Thus $g^+_s \to g^+_t$ u.c.\ on $R$ as $s \to t \in [0, t_0)$,
which yields the joint continuity of $g_t(z)$ on $[0, t_0) \times C^+_j$.
The joint continuity of $g_t(z)$ on $[0, t_0) \times (C^-_j \cup \{z_j, z^r_j\})$
is obtained in the same way.
\end{proof}

We now arrive at our main result.
Here, the dot $\dot{ }$ denotes the $t$-derivative.

\begin{thm} \label{thm:gKLeq}
Suppose that $\xi(t)$ is continuous and
$a_t$ is strictly increasing and differentiable over $[0, t_0)$.
Then the following are equivalent:
\begin{enumerate}
\item \label{cond:nicehull}
$\{F_t\}_{t \in [0, t_0)}$ is a family of continuously growing hulls in $D$
with driving function $\xi$ and half-plane capacity $a_t$.
\item \label{cond:gKLeq}
$\slit(t)$ and $g_t(z)$ solve the ODEs
\begin{gather} 
\frac{d}{dt}z_{j}(t) = -\pi \dot{a}_t \Psi_{\slit(t)}(z_{j}(t),\xi(t)), \quad
\frac{d}{dt}z^{r}_{j}(t) = -\pi \dot{a}_t \Psi_{\slit(t)}(z^{r}_{j}(t),\xi(t)),
\label{eq:gKLs} \\
\frac{d}{dt}g_{t}(z) = -\pi \dot{a}_t \Psi_{\slit(t)}(g_{t}(z),\xi(t)), \quad g_{0}(z)=z \in D.
\label{eq:gKL}
\end{gather}
\end{enumerate}
\end{thm}

\begin{proof}
It is sufficient to prove the theorem when $a_t=2t$,
for the general case is established by the time-change of $a_t$.
Under this half-plane capacity parametrization,
the ODEs \eqref{eq:gKLs} and \eqref{eq:gKL} reduce to
\eqref{eq:KLs2} and \eqref{eq:KL2}, respectively.

\noindent
$\eqref{cond:gKLeq} \Rightarrow \eqref{cond:nicehull}$:
Assume \eqref{cond:gKLeq}.
The conditions \eqref{eq:hull_limit} and $a_t=2t$ are already mentioned
in Section~\ref{subsec:initKL}.
To show the continuity of $\{F_t\}$,
observe that the continuity of $\slit(t)$ implies that of $\{D_t\}$
in the sense of kernel convergence.
It suffices to prove the joint continuity of $f_t(z):=g_t^{-1}(z)$ in $(t,z)$
because then $f_s \to f_t$ u.c.\ as $s \to t$ and
the kernel theorem~\ref{prop:kerthmH} implies
$D \setminus F_s \to D \setminus F_t$.

As a solution to the ODE~\eqref{eq:KL2}, $g_t(z)$ is jointly continuous.
We then see from Cauchy's integral formula that $g_t'(z)$ is also jointly continuous.
Note that it is non-vanishing since $g_t$ is univalent.
The joint continuity of $f_t(z)$ now follows from the fact that
it is a solution to the ODE
\[
\dot{f_t}(z)=2\pi f_t'(z)\Psi_{\slit(t)}(z, \xi(t))
	=\frac{2\pi}{g_t'(f_t(z))}\Psi_{\slit(t)}(z, \xi(t)).
\]

\noindent
$\eqref{cond:nicehull} \Rightarrow \eqref{cond:gKLeq}$:
We have already seen that \eqref{eq:KL2} holds
for $z \in D \cup \bigcup_j \partial_p C_j$
in the right derivative sense in Proposition~\ref{prop:KLright}~\eqref{prop:rdiff}.
Since $\xi(t)$, $\slit(t)$ and $g_t(z)$ are continuous in $t$
by \eqref{cond:nicehull} and Lemmas~\ref{lem:cont} and \ref{lem:jointcont},
\eqref{eq:KL2} holds as a genuine ODE
by the same proof as that of Theorems~9.8 and 9.9 in \cite{CFR16}.

To establish \eqref{eq:KLs2},
we can use the same method as in \cite[Section~2]{CF18}.
Hence it suffices to prove \cite[Lemma~2.1]{CF18} in our case,
because the proof of Lemma~2.2 and Theorem~2.3 in \cite{CF18}
depends only on this lemma, not on whether $F_t$ is a simple curve.
Assertions (i), (ii), (iv) and (v) of \cite[Lemma~2.1]{CF18} follows from
Cauchy's intergral formula and Lemma~\ref{lem:jointcont}.
The proofs of (iii), (vi) and (vii) are quite similar, and so we prove only (iii) here.

For a fixed $j$, let $g^+_t$ be the extended map of $g_t$ from $R_+$ to $R$
as in the proof of Lemma~\ref{lem:jointcont}.
We can check that $\eta_t(z, \xi_0):=\Psi_{\slit(t)}(g_t(z), \xi_0)$
can also be extended from $R_+$ to a holomorphic function on $R$,
which is denoted by $\eta^+_t(z, \xi_0)$,
and that $\eta^+_t(z, \xi_0)$ is continuous in $(t, z, \xi_0)$. 
By Proposition~\ref{prop:2ndorder} and Cauchy's integral formula we have,
for $0 \leq s < t < t_0$,
\[
\lvert (g^+_s)'(z)-(g^+_t)'(z)-2\pi (t-s) (\eta^+_s(z, \xi(s)))' \rvert
\leq 2(t-s) C' r_{s,t},
\]
where $r_{s,t}:=\inf\{R>0; g_s(F_t \setminus F_s) \subset B(\xi(s), R)\}$
and $C'$ is a constant.
Since $r_{s,t} \to 0$ as $t \downarrow s$, we obtain
\begin{equation} \label{eq:KLforderiv}
\frac{\partial^+(g^+_t)'(z)}{\partial t}=-2\pi (\eta^+_s(g^+_s(z), \xi(s)))'.
\end{equation}
The right hand side of \eqref{eq:KLforderiv} is jointly continuous in $(t,z)$
in view of (ii) of \cite[Lemma~2.1]{CF18}, and thus
the left hand side becomes the genuine derivative by \cite[Lemma~4.3]{La05}.
Therefore, $(g^+_t)'(z)$ is differentiable in $t$ and
$\partial_t(g^+_t)'(z)$ is continuous in $(t, z)$.

In this way, we can prove the assertions corresponding to \cite[Lemma~2.1]{CF18}
and then Lemma~2.2 and Theorem~2.3 of \cite{CF18}
tell us that \eqref{eq:KLs2} holds under
our assumption~\eqref{cond:nicehull} of the theorem.
\end{proof}

Every assertion in Section~\ref{subsec:deduction} remains valid
for the upper half-plane $\uhp$ in place of the standard slit domain $D$
by replacing $\Psi_D$ with $\Psi_{\uhp}=-\frac{1}{\pi}\frac{1}{z-\xi_0}$.

\subsection{Transformation of the chains, half-plane capacities and driving functions}
\label{subsec:trans}

From Theorem~\ref{thm:gKLeq}, the Komatu--Loewner evolution
defined in Section~\ref{subsec:initKL} proves to be
nothing but a family of continuously growing hulls
with continuous driving function and differentiable half-plane capacity $2t$.
In this subsection, we check that such nice properties on growing hulls
are independent of the domain $D$ and conformally invariant.
The transformation of the chains described in Section~\ref{sec:intro}
is thus always possible.
We take over the notations in Section~\ref{subsec:deduction}.

\begin{prop} \label{prop:cont}
$\{F_t\}$ is continuous with continuous driving function
and differentiable half-plane capacity in $D$
if and only if it has the same property in $\uhp$.
\end{prop}

\begin{proof}
Let $\iota \colon D \hookrightarrow \uhp$ be the inclusion map,
$g^0_t \colon \uhp \setminus F_t \to \uhp$, $t \in [0, t_0)$, be the canonical map
and $\iota_t := g^0_t \circ \iota \circ g_t^{-1}$.
By Schwarz's reflection, $\iota_t \colon D_t \hookrightarrow \uhp$ extends
to a conformal map from $D_t \cup \Pi D_t \cup \partial \uhp$
onto $g^0_t(D) \cup \Pi g^0_t(D) \cup \partial \uhp$.
It is especially a homeomorphism between these domains.

Assume that $\{F_t\}_{t \in [0,t_0)}$ is continuous
with continuous driving function $\xi$ and differentiable half-plane capacity $a_t$
in $D$.
We set $U(t):=\iota_t(\xi(t))$.
It holds that
\begin{align*}
\bigcap_{\delta>0}\overline{g^0_t(F_{t+\delta} \setminus F_t)}
&=\bigcap_{\delta>0}\overline{\iota_t \circ g_t(F_{t+\delta} \setminus F_t)} 
=\iota_t\left(\bigcap_{\delta>0}\overline{g_t(F_{t+\delta} \setminus F_t)}\right) \\
&=\iota_t(\{\xi(t)\}) =\{U(t)\}.
\end{align*}
Hence $\{F_t\}$ has driving function $U(t)$ in $\uhp$.
Next we fix $t \in (0, t_0)$.
For any $z \in \bigcup_j C_j$, there is some $r>0$ such that
$\overline{B(z, r)} \subset \uhp \setminus F_t \subset \uhp \setminus F_s$
for all $s \leq t$.
Combining this with the assumption that $D \setminus F_s \to D \setminus F_t$
as $s \uparrow t$,
we can conclude that $\uhp \setminus F_s \to \uhp \setminus F_t$.
This means the left continuity of $\{F_t\}$ in $\uhp$.
By Lemma~\ref{lem:cont}, $\{g^0_t\}$ is continuous in
the sense of uniform convergence on compacta.
$\iota_t(z)$ is then jointly continuous
in $(t,z) \in [0, t_0) \times (D_t \cup \Pi D_t \cup \partial \uhp)$.
Hence $U(t)=\iota_t(\xi(t))$ is continuous in $t$.
Finally the differentiability of $a^0_t:=\hcap^{\uhp}(F_t)$ is obtained from
the capacity comparison theorem \cite[Theorem~A.1]{CF18}.
Thus $\{F_t\}$ is continuous with continuous driving function $U$
and differentiable half-plane capacity $a^0_t$ in $\uhp$.

The proof of the converse implication is quite similar, and we omit it.
\end{proof}

Proposition~\ref{prop:cont} implies that,
if $\{F_t\}$ is continuous with continuous driving function and
differentiable half-plane capacity in one standard slit domain,
then we get the Komatu--Loewner equation on every standard slit domain.
In almost the same way,
we can also show that this condition is conformally invariant.
More precisely, \cite[(2.7)]{LSW01} combined with
the capacity comparison theorem \cite[Theorem~A.1]{CF18} yields the following:

\begin{thm} \label{thm:hcaptrans}
Denote by $D$ either a standard slit domain or the upper half-plane $\uhp$.
Denote also by $\tilde{D}$ a standard slit domain or $\uhp$,
but the degree of connectivity of $\tilde{D}$ can be different from that of $D$.
Let $\{F_t\}_{t \in [0, t_0)}$ be a family of continuously growing hulls
with continuous driving function $\xi$ and differentiable half-plane capacity $a_t$
in $D$.
Suppose that a domain $V \subset D$ and a univalent function
$h \colon V \hookrightarrow \tilde{D}$ satisfy the following conditions:
\begin{enumerate}
\item \label{cond:inclusion} $\bigcup_{t \in [0, t_0)}F_t \subset V$;
\item $\partial V \cap \partial \uhp$ is locally connected;
\item $h$ maps $\partial V \cap \partial \uhp$ into $\partial \uhp$, that is,
$\lim_{y \downarrow 0}\Im h(x+iy)=0$ for all $x \in \partial V \cap \partial \uhp$.
\end{enumerate}
Under these assumptions,
$\{h(F_t)\}$ is a family of continuously growing hulls in $\tilde{D}$.
Further let $g_t$ and $\tilde{g}_t$ be the canonical maps on $D \setminus F_t$ and
$\tilde{D} \setminus h(F_t)$, respectively, and set
$h_t := \tilde{g}_t \circ h \circ g_t^{-1}$ with the domain of definition
being $g_t(V \setminus F_t) \subset \uhp$.
By Schwarz's reflection, $h_t$ is extended to be holomorphic on
\[
g_t(V \setminus F_t) \cup \Pi g_t(V \setminus F_t) \cup \left(\partial \uhp
\setminus (\overline{D_t \setminus g_t(V \setminus F_t)})\right)\subset \C.
\]
$h_t(\xi(t))$ is then the continuous driving function of $\{h(F_t)\}$.
Moreover, $\tilde{a}_t:=\hcap^{\tilde{D}}(h(F_t))$ is differentiable with
\begin{equation} \label{eq:hcaptrans}
\dot{\tilde{a}}_t=h_t'(\xi(t))^2\dot{a}_t, \quad t \in [0, t_0).
\end{equation}
\end{thm}

\begin{figure}
\centering
\includegraphics[width=12cm]{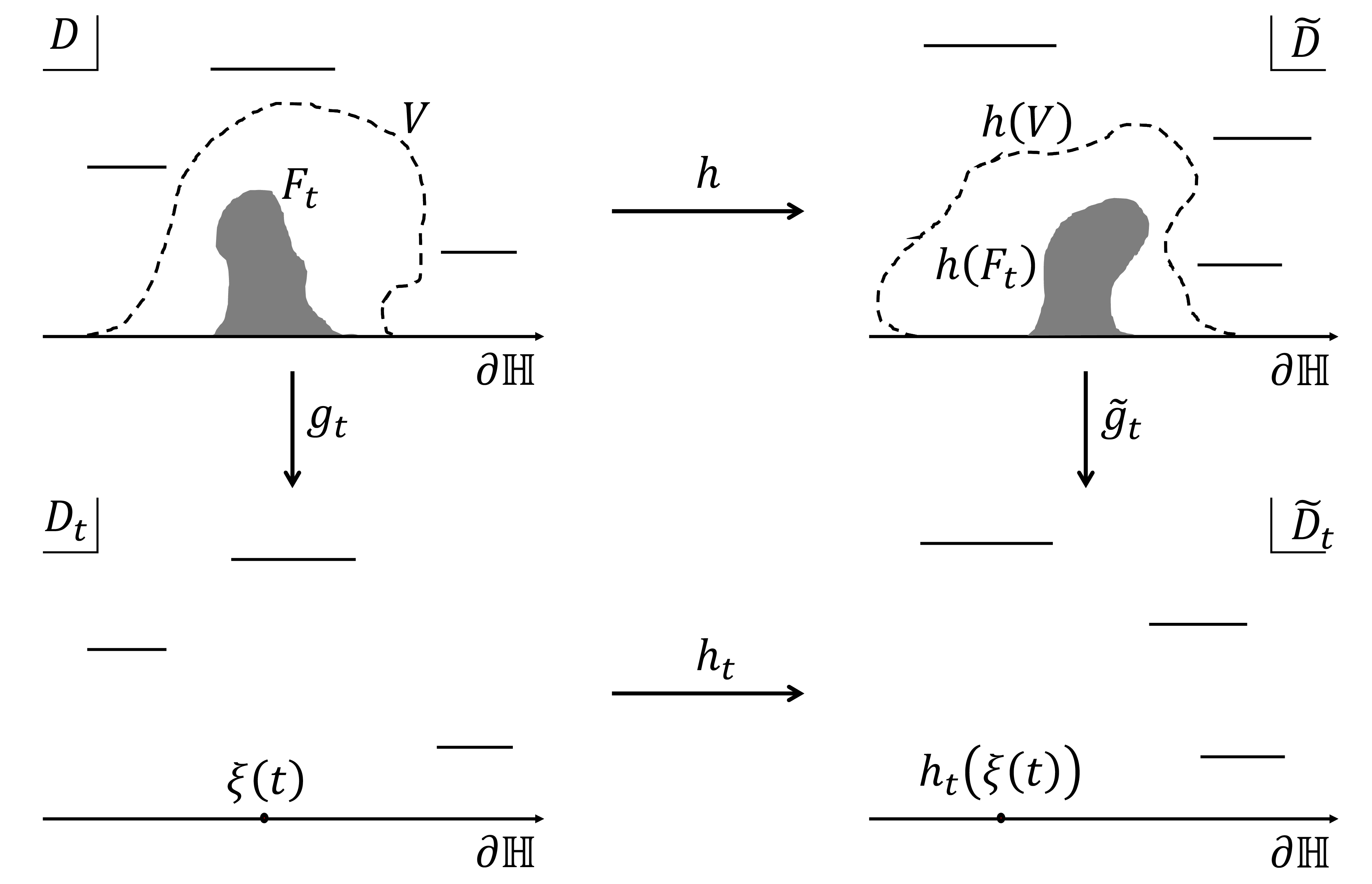}
\caption{Conformal maps $h$ and $h_t$}
\end{figure}

\begin{proof}
\eqref{eq:hcaptrans} can be shown in the same way as the proof of
\cite[Theorem~6.8]{CF18} by using the capacity comparison theorem in it.
Note that $h_t'(\xi(t)) \neq 0$ because $h_t$ is univalent.
The rest of the assertion can be shown in the same way as the proof of
Proposition~\ref{prop:cont} except that $\iota$ is replaced by $h$.
\end{proof}

We note that the degrees of connectivity of $D$ and $\tilde{D}$ can be different
in Theorem~\ref{thm:hcaptrans}.
Thus Theorems~\ref{thm:gKLeq} and \ref{thm:hcaptrans} establish
the transformation of Komatu--Loewner chains
under any possible conformal transformation.
More precisely, if $\{F_t\}$ is a Komatu--Loewner evolution $D$ driven by $\xi$,
then $\{h(F_t)\}$ in Theorem~\ref{thm:hcaptrans} is a family
of continuously growing hulls on $\tilde{D}$, and
we can reparametrize $\{F_t\}$ so that $\{h(F_t)\}$ obeys the
half-plane capacity parametrization on $\tilde{D}$ by setting
\begin{equation} \label{eq:hcaprep}
\check{F}_t:=h(F_{\tilde{a}^{-1}(2t)}),\ \check{g}_t:=\tilde{g}_{\tilde{a}^{-1}(2t)},\
\check{\xi}(t):=h_{\tilde{a}^{-1}(2t)}(\xi(\tilde{a}^{-1}(2t))).
\end{equation}
Now $(\check{g}_t, \check{F}_t)_{t \in [0, \tilde{a}_{t_0-}/2)}$ is
a Komatu--Loewner evolution in $\tilde{D}$ driven by $\check{\xi}$.
Note that the half-plane capacity is bounded if the hull is bounded
by \eqref{eq:hcap}.
Hence, the time-change $\tilde{a}_{\cdot}/2$ maps a compact subinterval
of $[0, t_0)$ onto a compact.
In particular, if $\bigcup_{t \in [0, t_0)} F_t$ is bounded,
then $\tilde{a}_{t_0-}<\infty$.

Finally, let $\{F_t\}$ be an $\skle_{\alpha, b}$ defined
at the end of Section~\ref{subsec:initKL}.
By Theorems~5.8 and 5.12 of \cite{CF18} and Theorem~\ref{thm:gKLeq}$,
\{F_t\}$ is a family of continuously growing hulls on $D$
driven by the solution $\xi(t)$ of the SDEs~\eqref{eq:driv} and \eqref{eq:slit}
with the half-plane capacity $a_t=2t$.
Under the setting of Theorem~\ref{thm:hcaptrans} on $D$, $\tilde{D}$, $V$ and $h$,
$\{h_t(F_t)\}$ becomes a family of continuously growing hulls in $\tilde{D}$
with the driving function $\tilde{\xi}(t)=h_t(\xi(t))$ and
with the half-plane capacity $\tilde{a}_t=2h_t'(\xi(t))^2t$.
Consequently, $\tilde{\slit}(t)$
with $D(\tilde{\slit}(t))=\tilde{g}_t(\tilde{D} \setminus h(F_t))$
and $\tilde{g}_t$ satisfy the ODEs~\eqref{eq:gKLs} and \eqref{eq:gKL}
for these choices of $\tilde{\xi}(t)$ and $\tilde{a}_t$ by Theorem~\ref{thm:gKLeq}
again. Denote these ODEs by \eqref{eq:gKLs}' and \eqref{eq:gKL}'.

In a similar manner to the proof of \cite[Theorem~6.9]{CF18},
one can then derive from \eqref{eq:gKL}'
the following semimartingale decomposition of the driving process
$\tilde{\xi}(t)=h_t(\xi(t))$ of $\{h(F_t)\}$:
\begin{align}
d\tilde{\xi}(t)&=
h_t'(\xi(t))\left(b(\slit(t)-\widehat{\xi}(t))+b_{\bmd}(\xi(t), \slit(t))\right)\,dt
\notag \\
&+\frac{1}{2}h_t''(\xi(t))\left(\alpha(\slit(t)-\widehat{\xi}(t))^2-6\right)\,dt
\label{eq:trans_semimart} \\
&-h_t'(\xi(t))^2b_{\bmd}(\tilde{\xi}(t), \tilde{\slit}(t))\,dt
+h_t'(\xi(t))\alpha(\slit(t)-\widehat{\xi}(t))\,dB_t, \quad t < T_V, \notag
\end{align}
where $T_V=\zeta \wedge \sup\{t>0; F_t \subset V\}$.
Here $b_{\bmd}(\xi_0, D)$ for $\xi_0 \in \partial \uhp$ and
for a standard slit domain $D$ is defined by
\[
b_{\bmd}(\xi_0, D) := 2\pi\lim_{z \to 0}
\left(\Psi_D(z, \xi_0)-\frac{1}{\pi}\frac{1}{ z-\xi_0}\right)
(=2\pi\mathbf{H}_D(\xi_0, \xi_0)),
\]
where $\mathbf{H}_D$ is the function appearing in \eqref{eq:PK}.
We let
\[
b_{\bmd}(\xi_0, \uhp)=0, \quad \xi_0 \in \partial \uhp
\]
accordingly.
We also write $b_{\bmd}(\xi_0, D)$ as $b_{\bmd}(\xi_0, \slit)$ in terms of
$\slit=\slit(D) \in \Slit$ for $D$.
We set $b_{\bmd}(\slit)=b_{\bmd}(0, \slit)$ and
call it the \emph{BMD domain constant} of $D=D(\slit)$.
By \cite[Lemma~6.1]{CF18}, the BMD domain constant $b_{\bmd}(\slit)$
is homogeneous with degree $-1$ and locally Lipschitz continuous so that
\[
b_{\bmd}(\xi_0, \slit)=b_{\bmd}(\slit-\widehat{\xi_0}),
\quad \slit \in \Slit, \xi_0 \in \partial \uhp.
\]

\eqref{eq:trans_semimart} is derived from the same computation
as in the proof of \cite[Theorem~6.9]{CF18} based on a generalized It\^o formula
\cite[Exercise~IV.3.12]{RY99}.
As pointed out in \cite[Remark~2.9]{CFS17},
one needs more assumptions than stated in \cite{RY99} to verify the formula.
Accordingly the following property of $h_t$ is necessary
to legitimize \eqref{eq:trans_semimart}:
\begin{enumerate}
\item[(C)] $h_t(z)$, $h_t'(z)$ and $h_t''(z)$ are jointly continuous
in $(t,z)$ for $z$ in a neighborhood of $\xi(t)$ in $\C$.
\end{enumerate}
We here prove Property (C) in a way similar to the proof
of Lemma~\ref{lem:jointcont}.

\begin{proof}[Proof of Property \rm{(C)}]
Fix $t \in [0, t_0)$.
Since $g_s \to g_t$ u.c.\ on $D \setminus F_t$,
we have $g_s \to g_t$ u.c.\ on $\Pi (D \setminus F_t)$
for the maps $g_s$ extended by Schwarz's reflection.
Since $\{g_s\}_s$ is a normal family on
$(D \setminus F_t) \cup\Pi (D \setminus F_t)
\cup (\partial \uhp \setminus \overline{F_t})$
by Corollary~\ref{cor:precptSigma} and Montel's theorem~\ref{prop:Montel},
$g_s \to g_t$ u.c.\ on this domain,
which we can check by the identity theorem
as in the proof of Lemma~\ref{lem:jointcont}.
Hence we can take a bounded open subset $U$ of $\C$ so that
\[
\xi(s) \in U \subset g_s(V \setminus F_s) \cup \Pi g_s(V \setminus F_s) \cup
\left(\partial \uhp \setminus (\overline{D_s \setminus g_s(V \setminus F_s)})\right),
\quad s \in [t_-, t_+],
\]
for some $0 \leq t_- \leq t < t_+ <t_0$.

We observe from the same argument as for $g_s$ that
$g_s^{-1} \to g_t^{-1}$ uniformly on $U$ as $s$ tends to $t$ in $[t_-, t_+]$
and that $\tilde{g}_s \to \tilde{g}_t$ u.c.\ on
$(\tilde{D} \setminus h(F_t)) \cup \Pi(\tilde{D} \setminus h(F_t))
\cup \left( \partial \uhp \setminus (\overline{\tilde{D} \setminus h(F_t)}) \right)$.
Thus the composite $h_s=\tilde{g}_s \circ h \circ g_s^{-1}$
converges to $h_t$ as $s \to t$ in $[t_-, t_+]$ uniformly on $U$.
As a consequence, Cauchy's integral formulae for $h_t'$ and $h_t''$
yield Property (C).
\end{proof}

Using the formula~\eqref{eq:trans_semimart} along with the ODEs~\eqref{eq:gKLs}'
and \eqref{eq:gKL}',
we arrive at the following theorem in exactly the same manner as
the proof of \cite[Theorem~6.11]{CF18}:

\begin{thm}[Conformal invariance of Chordal $\skle_{\sqrt{6},-b_{\bmd}}$]
\label{cor:locality}
Let $D$, $\tilde{D}$, $V$ and $h$ be as in Theorem~\ref{thm:hcaptrans}
except for Condition \eqref{cond:inclusion} for $V$.
For any $\skle_{\sqrt{6}, -b_{\bmd}}$ $\{F_t\}_{t \in [0, \zeta)}$ in $D$,
we set $T_V:=\zeta \wedge \sup\{t>0; F_t \subset V\}$.
Then, $\{\check{F}_t\}$ defined by \eqref{eq:hcaprep} is
$\skle_{\sqrt{6}, -b_{\bmd}}$ in $\tilde{D}$ up to $\tilde{a}(T_V-)/2$.
\end{thm}

Theorem~\ref{cor:locality} extends \cite[Theorem~4.2]{CFS17}
in which case $\tilde{D}=\uhp$ and $h$ is the inclusion map from $D$ into $\uhp$.
It also extends \cite[Theorem~6.11]{CF18} in which case
$h$ is the canonical map $\Phi_A$ from $D \setminus A$ for any hull $A$ in $D$
and $\tilde{D}=\Phi_A(D \setminus A)$.
The special case of \cite[Theorem~6.11]{CF18} where $D=\tilde{D}=\uhp$
and accordingly $b_{\bmd}=0$ was discovered
by Lawler, Schramm and Werner~\cite{LSW01, LSW03} and shown recently
in \cite{CFS17} more rigorously.
Such a property of $\sle_6$ has been called its \emph{locality}
under a phrase that \emph{$\sle_6$ does not feel the boundary before hitting it}.
Theorem~\ref{cor:locality} resolves some of the problems posed
in \cite[Section~5]{CFS17} as well.

\section*{Acknowledgments}
I would like to thank my supervisor Masanori Hino
for suggesting me to read \cite{CFR16} and
for many helpful conversations on my research.
I also wish to express my gratitude to Professors Masatoshi Fukushima and
Roland Friedrich for answering my questions kindly and
giving me a lot of valuable comments,
and to the referee for a careful reading of the paper.


\begin{thebibliography}{00}


\bibitem{BF08} R.\ O.\ Bauer and R.\ M.\ Friedrich, On chordal and bilateral SLE in multiply connected domains, Math.\ Z.\ {\bf 258} (2008), 241--265.
\bibitem{Ch12} Z.-Q.\ Chen, {\it Brownian Motion with Darning}, Lecture notes for talks given at RIMS, Kyoto University, 2012.
\bibitem{CF18} Z.-Q.\ Chen and M.\ Fukushima, Stochastic Komatu--Loewner evolutions and BMD domain constant, Stoch.\ Proc.\ Appl.\ {\bf 128} (2018), 545--594.
\bibitem{CFR16} Z.-Q.\ Chen, M.\ Fukushima and S.\ Rohde, Chordal Komatu--Loewner equation and Brownian motion with darning in multiply connected domains, Trans.\ Amer.\ Math.\ Soc.\ {\bf 368} (2016), 4065--4114.
\bibitem{CFS17} Z.-Q.\ Chen, M.\ Fukushima and H.\ Suzuki, Stochastic Komatu--Loewner evolutions and SLEs, Stoch.\ Proc.\ Appl.\ {\bf 127} (2017), 2068--2087.
\bibitem{Co95} J.\ B.\ Conway, {\it Functions of One Complex Variable II}, Graduate Texts in Mathematics, vol.\ 159, Springer-Verlag, New York, 1995.
\bibitem{Dr11} S. Drenning, Excursion reflected Brownian motion and Loewner equations in multiply connected domains, arXiv:1112.4123, 2011.
\bibitem{Fr10} R.\ Friedrich, The global geometry of stochastic Loewner evolutions, Adv.\ Stud.\ Pure Math., {\bf 57} (2010), 79--117.
\bibitem{Go69} G.\ M.\ Goluzin, {\it Geometric Theory of Functions of a Complex Variable}, Translation of Mathematical Monographs, vol. 26, American Mathematical Society, Providence, RI, 1969.
\bibitem{Ka15} M.\ Katori, {\it Bessel Processes, Schramm--Loewner Evolution, and the Dyson Model}, SpringerBriefs in Mathematical Physics, vol.\ 11, Springer, 2015.
\bibitem{Ko50} Y.\ Komatu, On conformal slit mapping of multiply-connected domains, Proc.\ Japan Acad. {\bf 26} (1950), 26--31.
\bibitem{La05} G.\ F.\ Lawler, {\it Conformally Invariant Processes in the Plane}, Mathematical Surveys and Monographs, vol.\ 114, American Mathematical Society, Providence, RI, 2005.
\bibitem{La06} G.\ F.\ Lawler, The Laplacian-$b$ random walk and the Schramm--Loewner evolution, Illinois J.\ Math.\ {\bf 50} (2006), 701--746.
\bibitem{LSW01} G.\ Lawler, O.\ Schramm and W.\ Werner, Values of Brownian intersection exponents, I: Half-plane exponents, Acta Math.\ {\bf 187} (2001), 237--273.
\bibitem{LSW03} G.\ Lawler, O.\ Schramm and W.\ Werner, Conformal restriction: the chordal case, J.\ Amer.\ Math.\ Soc.\ {\bf 16} (2003), 917--955.
\bibitem{Po75} C.\ Pommerenke, {\it Univalent Functions}, Vandenhoeck {\&} Ruprecht, G\"ottingen, 1975.
\bibitem{RY99} D.\ Revuz and M.\ Yor, {\it Continuous Martingales and Brownian Motion}, 3rd ed., Grundlehren der mathematischen Wissenschaften, vol.\ 293, Springer-Verlag, Berlin Heidelberg, 1999.
\bibitem{Sc00} O.\ Schramm, Scaling limits of loop-erased random walks and uniform spanning trees, Israel J.\ Math.\ {\bf 118} (2000), 221--288.
\bibitem{Ts59} M. Tsuji, {\it Potential Theory in Modern Function Theory}, Maruzen Co., Ltd., Tokyo, 1959.

\end{thebibliography}



\end{document}